%% file: main.tex
\documentclass[final,12pt]{colt2024} 

\title[Universal Lower Bounds and Optimal Rates]{
Universal Lower Bounds and Optimal Rates: Achieving Minimax Clustering Error in Sub-Exponential Mixture Models
} 
\usepackage{times}

\coltauthor{%
 \Name{Maximilien Dreveton} \Email{maximilien.dreveton@epfl.ch}\\
 \addr École Polytechnique Fédérale de Lausanne (EPFL)
 \AND
 \Name{Alperen G{\"o}zeten} \Email{alperengozeten@gmail.com}\\
 \addr Bilkent University%
 \AND 
 \Name{Matthias Grossglauser} \Email{matthias.grossglauser@epfl.ch}\\
 \addr École Polytechnique Fédérale de Lausanne (EPFL)
 \AND
 \Name{Patrick Thiran} \Email{patrick.thiran@epfl.ch}\\
 \addr École Polytechnique Fédérale de Lausanne (EPFL)
}

\newtheorem{condition}{Condition}

\usepackage{texmacro_maths}

\newcommand{\Chernoff}{ \mathrm{Chernoff} }
\newcommand{\loss}{ \mathrm{loss} }
\newcommand{\ham}{ \mathrm{Ham} }

\newcommand{\dbreg}{ \mathrm{Breg} }
\newcommand{\snr}{ \mathrm{SNR} }
\renewcommand{\dren}{ \mathrm{Ren} }

\newcommand{\tcZ}{ \tilde{\cZ} }

\newcommand{\oracle}{ \mathrm{oracle} }
\newcommand{\ideal}{ \mathrm{ideal} }
\newcommand{\lloyd}{ \mathrm{Lloyd} }
\newcommand{\excess}{ \mathrm{excess} }
\newcommand{\Lap}{ \mathrm{Lap} }

\newcommand{\wloss}{ \mathrm{w}\text{-}\loss  }

\newcommand{\Laplace}{ \mathrm{Laplace} }

\usepackage[bb=dsserif]{mathalpha}
\newcommand{\1}{ \mathbb{1} }

\usepackage{algorithm}

\begin{document}

\maketitle

\begin{abstract}
  Clustering is a pivotal challenge in unsupervised machine learning and is often investigated through the lens of mixture models. The optimal error rate for recovering cluster labels in Gaussian and sub-Gaussian mixture models involves \textit{ad hoc} signal-to-noise ratios. Simple iterative algorithms, such as Lloyd's algorithm, attain this optimal error rate. In this paper, we first establish a universal lower bound for the error rate in clustering \textit{any} mixture model, expressed through Chernoff information, a more versatile measure of model information than signal-to-noise ratios. We then demonstrate that iterative algorithms attain this lower bound in mixture models with sub-exponential tails, notably emphasizing location-scale mixtures featuring Laplace-distributed errors. Additionally, for datasets better modelled by Poisson or Negative Binomial mixtures, we study mixture models whose distributions belong to an exponential family. In such mixtures, we establish that Bregman hard clustering, a variant of Lloyd's algorithm employing a Bregman divergence, is rate optimal. 
\end{abstract}


\input{text.tex}


\bibliography{main.bib}

\clearpage
\appendix

\input{text_appendix}

\end{document}

%% file: text.tex
\section{Introduction}


 \new{Clustering} is the task of partitioning a set of data points into groups, called clusters, such that data points within the same cluster are more similar to each other than they are to data points in different clusters. Clustering is an important problem in statistics and machine learning~\citep{hastie2009elements,wu2008top}, with many applications. 

 \new{Mixture models} provide an elegant framework for the design and theoretical analysis of clustering algorithms~\citep{mclachlan2019finite,bouveyron2014model}. Denoting by $z^* \in [k]^n$ the vector of cluster assignments, a mixture model assumes that the $n$ observed data points $X_1, \cdots, X_n \in \cX^n$, where $\cX \subset \R^d$, are independently generated such that
 \begin{align}
 \label{eq:def_mixture_model}
     \forall i \in [n]  \colon X_i \cond z_i^* \sim f_{z_i^*}, 
 \end{align}
 where $f_1, \cdots, f_k$ are the $k$ probability distributions over $\cX$. An \new{estimator} $\hz = \hz_n$ of~$z^*$ is a measurable function $\hz \colon (X_1, \cdots, X_n) \mapsto \hz(X_1, \cdots, X_n) \in [k]^n$. 
 The loss of an estimator $\hz$ of $z^*$ is quantified by the number of disagreements between $\hz$ and $z^*$, up to a global permutation of $\hz$, \textit{i.e.,}
\begin{align}
 \label{eq:def_hamming_loss}
 \loss( z^*, \hz ) 
 \weq \min_{\tau \in \Sym(k)} \ham( z^*, \tau \circ \hz ),
\end{align}
where $\Sym(k)$ denotes the group of permutations on $[k]$ and $\ham$ the Hamming distance. The \new{expected relative error} made by an estimator $\hz \colon (X_1, \cdots, X_n) \to [k]^n$ is then defined as 
 \begin{align*} 
   \E \left[ n^{-1} \loss\left( z^*, \hz \right) \right], 
 \end{align*}
 where $\E[\cdot]$ denotes the expectation with respect to the model~\eqref{eq:def_mixture_model}. 
 
 \new{Gaussian mixture models} are an important class of mixture models, in which for all $a\in [k]$, the distribution $f_a$ is Gaussian with mean $\mu_a \in \R^p$ and covariance matrix $\Sigma_a$. 
 If the Gaussian mixture is \textit{isotropic}, that is, $\Sigma_a = \sigma^2 I_p$ for all $a \in [k]$ (with $\sigma^2 > 0)$, finding the maximum likelihood estimator (MLE) of $(z^*, \mu)$ is equivalent to solving the \new{k-means} problem 
 \begin{align}
 \label{eq:kmeans}
    \hz, \hmu \weq \argmin_{ \substack{ z \in [k]^n \\ \tmu_1, \cdots, \tmu_p \in \R^{p} } } \sum_{ i = 1 }^n \sum_{a=1}^k \1\{ z_i = a \} \| X_i - \tmu_{z_a} \|^2_2,
 \end{align}
 where $\| \cdot \|_2$ denotes the Euclidean $\ell^2$-norm. 
 \new{Lloyd's algorithm} provides a simple way to find an approximate solution of this NP-hard minimisation problem iteratively~\citep{lloyd1982least}.
 Under technical conditions on the initialisation and on the model parameters,~\cite{lu2016statistical} show that the number of misclustered points by Lloyd's algorithm after $\Theta(\log n)$ iterations verify
 \begin{align}
 \label{eq:error_snr}
   \E \left[ n^{-1} \loss\left( z^*, \hz^{\lloyd} \right) \right] \wle e^{ -(1+o(1)) \frac12 \snr^2 },
 \end{align}
 where the signal-to-noise ratio (SNR) of this isotropic Gaussian mixture is defined as 
 \begin{align}
 \label{eq:def_snr}
   \snr \weq \frac{ \min_{ 1 \le a \ne b \le k} \| \mu_a - \mu_b \|_2 }{2 \sigma}.
 \end{align}
 More generally, in a mixture of anisotropic Gaussians, where the probability distributions $f_1, \cdots, f_k$ are Gaussians with means $\mu_1, \cdots, \mu_k$ and share the same covariance matrix $\Sigma$, the MLE becomes
 \begin{align*}
    \hz, \hmu \weq \argmin_{ \substack{ z \in [k]^n \\ \tmu_1, \cdots, \tmu_p \in \R^{p} } } \sum_{ i = 1 }^n \sum_{a=1}^k \1\{ z_i = a \} \| X_i - \tmu_{z_a} \|_{\Sigma}^2,
 \end{align*}
 where $\|x-y\|_{\Sigma}^2 = (x-y)^T \Sigma^{-1} (x-y)$ denotes the square of the Mahalanobis distance. 
 Because the formulation of the new MLE is similar to~\eqref{eq:kmeans}, it is natural to modify the Lloyd algorithm by (i) estimating the means and covariances in the estimation step and (ii) replacing the squared Euclidean distance by the Mahalanobis distance in the clustering step. The error of this new iterative algorithm is also upper-bounded as in~\eqref{eq:error_snr}, where the signal-to-noise ratio of this model is $\snr_{\mathrm{anisotropic}} = 2^{-1} \min_{1 \le a \ne b \le k} \| \mu_a - \mu_b \|_{\Sigma} 
 $~\citep{chen2021optimal}. 
 Moreover, for this anisotropic Gaussian mixture model, the upper bound~\eqref{eq:error_snr} on the error attained by this iterative algorithm is tight and cannot be improved. More precisely,~\cite{chen2021optimal} establish that 
 \begin{align}
 \label{eq:lower_bound_error_GMM}
    \inf_{\hz} \sup_{z^* \in [k]^n} \E \left[ n^{-1} \loss\left( z^*, \hz \right) \right] \wge e^{ -(1+o(1)) \frac12 \snr^2_{\mathrm{anisotropic}} }, 
 \end{align}
 where the $\inf$ is taken over all estimators $\hz$. Combining~\eqref{eq:error_snr} and~\eqref{eq:lower_bound_error_GMM} shows that a Lloyd-based scheme for solving the MLE in a mixture of anisotropic Gaussians achieves the minimax rate. 
 
 Error rates of iterative algorithms are now well understood for (sub)-Gaussian mixture models, but \textit{what happens for the mixture of distributions with heavier tails than Gaussians?} 
 As a first setting, we will study \new{location-scale mixture models}, for which the coordinates of each data point are generated as  
 \begin{align}
 \label{eq:location_scale_mixture_model}
    \forall i \in [n], \, \forall \ell \in [d] \colon \quad X_{i\ell} \weq \mu_{z_i^* \ell } + \sigma_{z_i^* \ell} \epsilon_{i\ell}, 
 \end{align}
 where $(\mu_a, \sigma_a) \in \R^d \times (0,\infty)^{d}$ denote the location and the scale of block $a \in [k]$, and the random variables $\epsilon_{i\ell}$ are independently sampled from a distribution with mean 0 and variance $1$. \textit{When the~$\epsilon_{i\ell}$'s are not Gaussians, does the minimax error rate also relate to a signal-to-noise ratio?} Furthermore, parametric mixture models do not always involve location and scale parameters. For example, single-cell RNA sequencing datasets are represented by a matrix $X \in \Z^{n \times d}$ where $n$ is the number of cells, $d$ is the number of genes, and $X_{i\ell}$ records the number of unique molecular identifiers from the $i$-th cell that map to the $\ell$-th gene. Cells can be of different types, and entries $X_{ij}$ are often assumed to come from a negative binomial, whose parameters depend on the types of the cell~$i$ and of the gene $\ell$~\citep{grun2014validation,haque2017practical}. This motivates the study of negative binomial mixture models and, more generally, of mixture models of the form~\eqref{eq:def_mixture_model} where the pdfs $f_1, \cdots, f_k$ belong to an exponential family, but not necessarily Gaussian.

 
 Our first contribution is the characterisation of a fundamental limit for the misclustering error. Denote by $\Chernoff(f,g)$ the \new{Chernoff information} between two probability distributions $f$ and $g$. We establish that the classification error made by \textit{any} clustering algorithm when applied to a mixture model defined in~\eqref{eq:def_mixture_model} is lower bounded as   
 \begin{align}
  \label{eq:lower_bound_error_MixtureModels}
    \inf_{\hz} \sup_{ z^* \in [k]^n } \E \left[ n^{-1} \loss \left( z^*, \hz \right) \right] \wge e^{-(1+o(1)) \min\limits_{1 \le a \ne b \le k} \Chernoff(f_a, f_b)}. 
 \end{align}
 This lower bound involves the Chernoff information instead of signal-to-noise ratios, but we show that these two quantities are related in many models of interest. In particular, for anisotropic Gaussian mixture models, the lower bounds~\eqref{eq:lower_bound_error_MixtureModels} and~\eqref{eq:lower_bound_error_GMM} are the same. However, expressing the lower bound in terms of the Chernoff information instead of SNR is more versatile as it does not require making any assumption on the pdfs $f_1,\cdots,f_k$. 
 The rationale for finding the Chernoff information in the lower bound~\eqref{eq:lower_bound_error_MixtureModels} lies in the reformulation of the problem of assigning a data point $X_i$ to a cluster $\hz_i$ as an equivalent hypothesis testing problem, which tests the $k$ different hypothesis $H_k \colon \hz_i = \ell$ for $\ell \in [k]$. The difficulty of this problem is defined as the error made by the best estimator, which is asymptotically $\exp(-(1+o(1)) \min_{1 \le a \ne b \le k} \Chernoff(f_a, f_b))$. 

 Our second contribution is to show that iterative clustering algorithms can attain this error rate in sub-exponential mixture models. More precisely, we establish that an iterative algorithm achieves the lower bound~\eqref{eq:lower_bound_error_MixtureModels} in the location-scale mixture model~\eqref{eq:location_scale_mixture_model} when the $\epsilon_{i\ell}$'s are Laplace-distributed. An interesting example is when each dimension has the same variance, \textit{i.e.,} $\sigma_{1\ell} = \cdots = \sigma_{k\ell} =: \sigma_{\ell}$. In such a model, the minimax error rate can be written as $\exp( -(1+o(1)) \snr_{\Laplace} )$, where 
 \begin{align*}
    \snr_{\Laplace} \weq \min_{1 \le a \ne b \le k } \| \Sigma^{-1} (\mu_a - \mu_b) \|_1, 
 \end{align*}
 where $\Sigma$ is the diagonal matrix whose diagonal elements are $\sigma_1, \cdots, \sigma_d$. 
 Furthermore, for the mixture model defined in~\eqref{eq:def_mixture_model} whose pdfs belong to an exponential family, we show that the lower bound~\eqref{eq:lower_bound_error_MixtureModels} is attained by a variant of Lloyd's algorithm that replaces the minimisation of the squared Euclidean distance by the minimization of a Bregman divergence. This algorithm is commonly called \new{Bregman hard clustering} in the literature, and the choice of the Bregman divergence depends on the exponential family considered~\citep{banerjee2005clustering}. 

 \paragraph{Paper structure} The paper is structured as follows. In Section~\ref{sec:lower_bound}, we establish a lower bound on the error rate made by any algorithm in clustering mixture models. We show in Section~\ref{sec:upper_bound} that iterative algorithms attain this lower bound in various mixture models, such as the Laplace mixture model (Section~\ref{subsection:LaplaceMixtures}) and the exponential family mixture models (Section~\ref{subsection:exponentialFamilyMixtures}).
 We discuss these results in Section~\ref{sec:discussion} and compare them with the existing literature.


\paragraph{Notations}

The notation $1_n$ denotes the vector of size $n\times 1$ whose entries are all equal to one. For a vector $x$, we denote by $\|x\|_p$ its $\ell^p$ norm (with $1 \le p \le \infty$). The standard scalar product between two vectors $x, y$ is denoted $<x,y>$. 
The indicator of an event $A$ is denoted $\1\{A\}$. We abbreviate "random variable" by rv and "probability density function" by pdf. Laplace and Gaussian random variables with mean $\mu$ and scale $\sigma$ are denoted by $\Lap(\mu,\sigma)$ and $\Nor(\mu,\sigma^2)$. Given a pdf $f$, we write $X \sim f$ if $X$ is a rv whose pdf is $f$. A real-valued rv $X$ is \textit{sub-exponential} if there exists $C>0$ such that for all $x\ge 0$ we have $\pr\left( |X| \ge x \right) \le 2e^{-Ct}$. 
Finally, we use the standard Landau notations $o$ and $O$, and write $a = \omega(b)$ when $b = o(a)$ and $a = \Omega(b)$ when $b = O(a)$. We also write $a= \Theta(b)$ when $a=O(b)$ and $b=O(a)$. 

\section{Minimax rate of the clustering error in mixture models}
\label{sec:lower_bound}

Let $f$ and $g$ be two pdfs with respect to a reference dominating measure $\nu$. 
 The \new{Chernoff information} between $f$ and $g$ is defined as 
\begin{align*}
 \Chernoff(f,g) \weq - \log \inf_{t \in (0,1)} \int f^{t}(x) g^{1-t}(x) d\nu(x).
\end{align*}
For a family $\cF = (f_1, \cdots, f_k)$ composed of $k$ probability distributions, we define 
\begin{align}
\label{eq:def_chernoff_family}
 \Chernoff( \cF ) \weq \min_{ 1 \le a \ne b \le k } \Chernoff \left( f_{a}, f_{b} \right). 
\end{align}
The following theorem establishes an asymptotic lower bound on the clustering error in a mixture model composed of the distributions belonging to the family $\cF$.  
\begin{theorem}
\label{thm:lower_bound_error}
Consider the mixture model defined in~\eqref{eq:def_mixture_model} and let $\cF = (f_{1}, \cdots, f_k)$ be the family of $k$ probability distributions that comprise the mixture, where $k$ and the distributions $f_a$ scale with $n$. 
Suppose that $\Chernoff( \cF ) = \omega( \log k )$. Then,  
 \begin{align*}
   \inf_{\hz} \sup_{ \substack{ z \in [k]^n } } \E \left[ n^{-1} \loss(z, \hz) \right] 
   \wge e^{ - (1+o(1)) \Chernoff( \cF ) }, 
 \end{align*}
 where the $\inf$ is taken over all estimators $\hz \colon (X_1, \cdots, X_n) \to [k]^n$. 
\end{theorem}

The proof of Theorem~\ref{thm:lower_bound_error} is given in Section~\ref{section:proof_lowerBound}. 
The proof of Theorem~1 has two main steps. The first challenge is to address the minimum over all permutations in the definition of the error loss. Hence, rather than directly examining $\inf_{\hat{z}} \sup_{z \in [k]^n } \E \left[ n^{-1} \loss(z, \hat{z}) \right]$, we focus on a sub-problem $\inf_{\hat{z}} \sup_{z \in \tcZ} \E \left[ n^{-1} \loss(z, \hat{z}) \right]$, where $\tcZ \subset [k]^n$ is chosen such that $\text{loss}(z, \hat{z}) = \ham(z, \hat{z})$ for all $z, \hat{z} \in \tcZ$. The idea is that this sub-problem is simple enough to analyze, but it still captures the hardness of the original clustering problem. 
Next, we bound the minimax risk of this sub-problem by the Bayes risk and we demonstrate that it is sufficient to lower-bound the testing error between each pair. The optimal error of this pairwise testing problem follows naturally from Lemma~\ref{lemma:hypothesis_testing}.

In Gaussian mixture models, the following example shows that the quantity $\Chernoff(\cF)$ is related to the more commonly used signal-to-noise ratios. 

\begin{example}
 \label{example:anisotropicGMM}
 Suppose that $f_a = \Nor(\mu_a, \Sigma_a)$ where $\mu_1, \cdots, \mu_k \in \R^d$ and $\Sigma_1,\cdots,\Sigma_k$ are $k$-by-$k$ positive definite matrices. Let $\Sigma_{abt} = (1-t) \Sigma_a + t \Sigma_b$. Then, 
\begin{align*}
 \Chernoff(\cF) = \min_{a \ne b \in [k]} \sup_{t \in (0,1)} (1-t) \left\{ \frac{t}{2}( \mu_a - \mu_b)^T \Sigma_{abt}^{-1} ( \mu_a - \mu_b ) - \frac{1}{2(t-1)} \log \frac{ | t \Sigma_a + (1-t) \Sigma_b |}{ | \Sigma_a |^{1-t} \cdot | \Sigma_b |^t} \right\}. 
\end{align*}
 When $\Sigma_1 = \cdots = \Sigma_k$ the $\sup$ is achieved for $t=2^{-1}$ and we obtain $\Chernoff(\cF) = 2^{-1} \snr^2_{\mathrm{anisotropic}}$ where $\snr_{\mathrm{anisotropic}} = 2^{-1} \min_{a, b \in [k]} \| \Sigma^{-1/2} (\mu_a-\mu_b)  \|_2$. This recovers the minimax lower bound for clustering Gaussian mixtures established in~\cite{chen2021optimal}. 
\end{example}

 Theorem~\ref{thm:lower_bound_error} is closely related to hypothesis testing. Indeed, suppose that the probability densities $f_1, \cdots, f_k$ are known. By the Neyman-Pearson lemma, the optimal clustering $\hz^{\MLE}$ verifies
\begin{align*}
 \hz_i^{\MLE} \weq \argmax_{a \in [k]} f_a(X_i), 
\end{align*}
 and $\hz^{\MLE}_i$ is a function of $X_i$ only, and not of the other data points $X_{-i} = (X_j)_{j \ne i}$. 

 Yet, hypothesis testing conventionally operates within the framework of \textit{fixed} pdfs $f$ and $g$, where observations consist of $n$ data points $Y_1, \cdots, Y_n$, independently sampled from either $f$ or~$g$. It is standard to quantify the optimal error rate of this problem using the Chernoff information. We focus on an alternative scenario: when we have two sequences of distributions $f_m$ and $g_m$, indexed by a parameter $m$, which diverge with $m$ (as indicated by the unbounded Chernoff information), distinguishing between the two hypotheses at each iteration $m$ becomes feasible with just a \textit{single} data point~$Y$ sampled from $f_m$ or $g_m$. The following lemma, whose proof is given in Appendix~\ref{proof:lemma:hypothesis_testing}, provides the optimal error rate of this hypothesis problem. 
 This lemma cannot be directly derived from existing results, as the pdfs $f$ and~$g$ are not fixed but vary with $m$.


\begin{lemma}
\label{lemma:hypothesis_testing}
 Given two sequences of pdfs $(f_m)$ and $(g_m)$ indexed by a parameter $m \in \Z_+$, consider the two hypotheses $H_0 \colon Y \sim f_m$ and $H_1 \colon Y \sim g_m$. Let 
$
\phi^{\MLE}(Y) \weq \1\{ f_m(Y) < g_m(Y) \}  
$
 and define the worst-case error of $\phi \colon Y \mapsto \phi(Y) \in \{0,1\}$ by 
 \[
  r(\phi) = \max \{ \pr \left( \phi(Y) = 0 \cond H_1 \right), \pr \left( \phi(Y) = 1 \cond H_0 \right) \}. 
 \] 
 Then, $\inf_{\phi} r(\phi) = r(\phi^{\MLE})$. Furthermore, if $\Chernoff(f_m, g_m,) = \omega(1)$, we have 
 \begin{align*}
    \log r(\phi^{\MLE}) \weq - (1+o(1)) \Chernoff(f_m, g_m).
 \end{align*}
 Otherwise, if $\Chernoff(f_m, g_m) = O(1)$, we have 
$
r(\phi^{\MLE}) \ge c \
$ 
 for some constant $c > 0$. 
\end{lemma}

 A direct consequence of Lemma~\ref{lemma:hypothesis_testing} is that the classification rule $\hz_i = \argmax_{ m \in [k]} f_m(X_i)$ for all $i\in[n]$ yields an error rate of $\exp(-(1+o(1)) \Chernoff(\cF))$. 
 Hence, if one has access to the true probability distributions $f_1, \cdots, f_k$ composing the mixture, then the lower bound given in Theorem~\ref{thm:lower_bound_error} is tight. In most practical settings, the true probability distributions are unknown. The following section demonstrates how the minimax error rate can still be achieved.

\section{Clustering error of iterative algorithms on parametric mixture models}
\label{sec:upper_bound}

 \subsection{Parametric mixture model}
 \label{subsection:parametricMixtureModel}
 
 In this section, we consider the recovery of the clusters of parametric mixture models. More precisely, we let $\cP_{\Theta} = \{ f_{\theta}, \theta \in \Theta \}$ be a family of pdfs parameterised by a subset $\Theta\subset \R^p$. For $m$ points $Y_1, \cdots, Y_m$ sampled from $f_{\theta}$, we denote by $\htheta( \{ Y_1, \cdots, Y_m \} )$ an estimator of $\theta$. 
 We let $z^* \in [k]^n$ be the vector of cluster assignments, and $\theta_1, \cdots, \theta_k \in \Theta$. Conditioned on $z^*$, the $n$ observed data points $(X_1, \cdots, X_n)$ are independently sampled, such that
 \begin{align}
    X_i \cond z_i^* = a \wsim f_{\theta_a}.
 \end{align}
 A natural estimator $\hz^{\oracle}$ of $z^*$ that uses the knowledge of the true parameters of the model $(\theta_a)_{a \in [k]}$ is given by 
 \begin{align}
 \label{eq:def_oracle_clusteringProcedure}
  \forall i \in [n] \colon \quad \hz_i^{\oracle} \weq \argmax_{ a \in [k] } \log f_{\theta_a} (X_i).
 \end{align}
 When the model parameters $(\theta_a)_{a \in [k]}$ are unknown, Algorithm~\ref{algo:iterativeClustering_parametricMixture} provides an iterative scheme for estimating the cluster assignment $z^*$. This algorithm sequentially performs the estimation and clustering stages. The goal of this section is to provide general bounds on the error made by Algorithm~\ref{algo:iterativeClustering_parametricMixture}. 
 Following the same proof strategy as previous works on iterative algorithms~\citep{gao2022iterative}, we first decompose the error into two terms and then provide general conditions under which these terms can be upper-bounded.
 
\begin{algorithm}[!ht] 
 \caption{Clustering parametric mixture models.}
\label{algo:iterativeClustering_parametricMixture}
 \KwInput{ Set of $n$ data points $(X_1, \cdots, X_n) \in \cX^{n}$, parametric family $\cP_{\Theta} = \{ f_{\theta}, \theta \in \Theta \}$ of pdfs, number of clusters $k$, number of iteration $t_{\max}$, initial clustering $\hz^{(0)} \in [k]^n$.}
 \KwOutput{Predicted clusters $\hz \in [k]^n$.}

 \textbf{For} $t = 1 \cdots t_{\max}$ \textbf{do}
  \begin{enumerate}
    \item For $a= 1, \cdots, k$, let $\htheta_a^{(t)} = \htheta\left( \{ X_i \colon z_i^{(t-1)} = a \} \right)$ be an estimate of $\theta_a$; 
    \item For $i=1, \cdots, n$ let $\hz_i^{(t)} = \argmax_{ a \in [k] } \log f_{\htheta_a^{(t)}}(X_i)$. 
  \end{enumerate}
 \KwReturn{$\hz = \hz^{(t_{\max})}$.}
\end{algorithm}

 \subsubsection{Decomposition of the error term}
 Let us introduce $\ell_a(x) = \log f_{\theta_a}(x)$ and $\hell_a^{(t)}(x) = \log f_{\htheta_a^{(t)}}(x)$. We show in Appendix~\ref{appendix:subsection:decompositionError} that we can upper-bound the error $\loss(z^*, \hz^{(t)})$ of Algorithm~\ref{algo:iterativeClustering_parametricMixture} made at step $t$ as
 \begin{align}
 \label{eq:split_loss_ideal_excess}
 \loss\left( z^*, \hz^{(t)} \right) 
 & \wle \xi_{\ideal}(\delta) + \xi_{\excess}^{(t)}(\delta), 
\end{align}
where $\delta > 0$ and 
\begin{align}
 \xi_{\ideal}(\delta) & \weq \sum_{i \in [n]} \sum_{ \substack{ b \in [k] \backslash \{z^*_i\} } } \1 \left\{ \ell_{z_i^*}(X_i) - \ell_{_b}(X_i) < \delta \right\}, 
 \label{eq:def_ideal_error} \\ 
 \xi_{\excess}^{(t)}(\delta) & \weq 2 \delta^{-1} \sum_{i \in [n]} \sum_{ \substack{ b \in [k] \backslash \{z^*_i\} } } \1 \left\{ \hz_i^{(t)} = b \right\} \max_{a \in [k]} \left| \hell_{a}^{(t)}(X_i) - \ell_a(X_i) \right|. 
 \label{eq:def_excess_error}
\end{align}
 When $\delta=0$, the ideal error $\xi_{\ideal}(0)$ is an upper bound on the error done by one step of Algorithm~\ref{algo:iterativeClustering_parametricMixture} that uses the correct parameters $\theta^*_1, \cdots, \theta_k^*$ and not the estimated ones. 
 Studying $\xi_{\ideal}(\delta)$ instead of $\xi_{\ideal}(0)$ gives us some room to control the excess error $\xi_{\excess}^{(t)}(\delta)$ made by estimating the model parameters. The value of $\delta$ must be small enough so that $\xi_{\ideal}(\delta)$ has the same asymptotic behaviour as $\xi_{\ideal}(0)$, but large enough so that $\xi_{\excess}^{(t)}(\delta)$ remains small. 
 The following lemma motivates the choice of $\delta = o(\Chernoff(\cF))$. 

 \begin{lemma}
  \label{lemma:ideal_error_rate}
  Consider a family $\cF = (f_1, \cdots, f_k)$ of pdf. Suppose $\Chernoff(\cF) = \omega( 1 )$ and let $\delta = o( \Chernoff(\cF))$. Then, with a probability of at least $1 - e^{- \sqrt{\Chernoff(\cF)}}$, 
    \begin{align*}
     \xi_{\ideal}(\delta) \wle n k e^{ -(1+o(1)) \Chernoff( \cF ) }.
    \end{align*}
 \end{lemma}

\subsubsection{Conditions for recovery}

 After Lemma~\ref{lemma:ideal_error_rate}, the last remaining step to upper-bound the loss is to control the excess error term. Because the estimates $\hz^{(t)}$ are data dependent, we have to establish that, starting from \textit{any} $\hz^{(0)}$ with a loss small enough, the excess error after one step is upper bounded by a nicely behaved quantity. More precisely, denote $z_{\mathrm{new}}$ the clustering obtained after one step of Algorithm~\ref{algo:iterativeClustering_parametricMixture} starting from some arbitrary initial configuration $z_{\mathrm{old}} \in [k]^n$, and define the following event 
 \begin{align*}
    \cE_{\tau,\delta,c,c'} \weq & 
    \Big\{ \loss( z^*, z_{\mathrm{old}} ) \le n k^{-1} \tau \ \text{ implies } \  \xi_{\excess}(\delta) \le c \, \loss ( z^*, z_{\mathrm{new}} ) + c' \, \loss( z^*, z_{\mathrm{old}} )  \Big\},
 \end{align*} 
 where $\tau,\delta,c,c'$ are determined later. The following condition states that the event $\cE_{\tau,\delta,c,c'}$ holds with probability $1-o(1)$ (with respect to the data sampling process) for a certain choice of $\tau, \delta$. 

 \begin{condition}
 \label{condition:excess_error} 
 Assume there exists $\tau = \Omega(1)$, $\delta = o(\Chernoff(\cF))$ and constants $c, c' \in (0,1)$ with $c' < 1-c$ such that $\pr\left( \cE_{\tau,\delta,c,c'} \right) \ge 1 - o(1)$. 
\end{condition}
 Assume $\loss(z^*, \hz^{(0)} ) \le nk^{-1} \tau$. Conditionally on the high probability event $\cE_{\tau,\delta,c,c'}$, we establish (by induction and by combining the error decomposition~\eqref{eq:split_loss_ideal_excess} with Lemma~\ref{lemma:ideal_error_rate}) that 
 \begin{align}
 \label{eq:loss_linear_convergence_rate}
    \loss \left( z^*, \hz^{(t)} \right) \wle \frac{nk}{1-c} e^{ -(1+o(1)) \Chernoff( \cF ) } + \frac{c'}{1-c} \loss \left( z^*, \hz^{(t-1)} \right),
 \end{align}
 as long as we can ensure that $\loss\left( z^*, \hz^{(t)} \right) \le n k^{-1} \tau$ at every step $t \ge 0$ for the \textit{same} $\tau = O(1)$. We can now state the following lemma, whose proof is given in Appendix~\ref{appendix:proof_thm:generalParametricMixtureRecovery}. 

\begin{lemma}
\label{lemma:generalParametricMixtureRecovery}
  Let $\theta_1, \cdots, \theta_k \in \Theta$ and $\cF = (f_{\theta_1}, \cdots, f_{\theta_k})$. Let $\tau = \Omega(1)$ such that Condition~\ref{condition:excess_error} holds and $\Chernoff(\cF) = \omega( \log (k^2 \tau) )$. Let $\hz^{(t)}$ be the output of Algorithm~\ref{algo:iterativeClustering_parametricMixture} after $t$ steps. 
  We have 
  \begin{align*}
    \forall t \ge \left \lfloor \log \left( \frac{1-c}{c'} \right) \log n \right \rfloor \colon \qquad n^{-1} \loss\left( z^*, \hz^{(t)} \right) \wle e^{ -(1+o(1)) \Chernoff(\cF)}. 
  \end{align*}
\end{lemma}
 Lemma~\ref{lemma:generalParametricMixtureRecovery} establishes that  Algorithm~\ref{algo:iterativeClustering_parametricMixture} achieves the minimax rate of recovering $z^*$ with respect to the loss function $n^{-1} \loss(z^*, z)$ after at most $\Theta( \log n )$ iterations when Condition~\ref{condition:excess_error} is verified. In the next two sections, we show that this condition holds for specific parametric families $\cP(\Theta)$.

 \subsection{Clustering Laplace mixture models}
 \label{subsection:LaplaceMixtures}

 A real-valued random variable $Y$ has a (1-dimensional) Laplace distribution with location $\mu \in \R$ and scale~$\sigma > 0$ if its pdf is 
 $g_{ (\mu, \sigma) }(x) = \frac{1}{2\sigma} \exp\left( - \sigma^{-1} |x-\mu| \right)$. We denote such a rv by $Y \sim \Lap(\mu, \sigma)$. In this section, we suppose that the $n$ observed data points $X_1, \cdots, X_n$ belong to $\R^d$ and are generated from the mixture model~\eqref{eq:def_mixture_model} such that for every $i$, the $d$ coordinates of $X_i$ are independently generated and follow a Laplace distribution, \textit{i.e.,}
\begin{align}
 \label{eq:location_scale_mixture}
 \forall \ell \in [d] \colon X_{i\ell} \weq \mu_{z_i^* \ell} + \sigma_{z_i^* \ell} \epsilon_{i \ell} \quad \text{ where } \epsilon_{i \ell} \wsim \Lap(0,1), 
\end{align}
where for all $a\in [k]$ we have $\mu_a \in \R^d$ and $\sigma_{a} \in (0,\infty)^{d}$. 
Equivalently, we can rewrite the mixture~\eqref{eq:location_scale_mixture} as a mixture of the parametric family indexed over $\Theta = \R^d \times (0, \infty)^{d}$ defined by  
\[
 \cP(\Theta) \weq \left\{ f_{\theta}(x) = \prod_{\ell=1}^d \sigma_{\ell}^{-1} g_{(0,1)} \left( \frac{x_{\ell} - \mu_\ell}{ \sigma_{\ell} } \right), \theta = (\mu, \sigma) \right\}. 
\]
Given a sample $Y_1, \cdots, Y_m$ of a 1-dimensional Laplace distribution, we estimate the location and the scale by 
\begin{align*}
 \hmu(Y_1,\cdots,Y_m) \weq m^{-1} \sum_{i=1}^m Y_i, 
 \quad \text{ and } \quad 
 \hsigma(Y_1, \cdots, Y_m) \weq m^{-1} \sum_{i=1}^m \left| Y_i - \hmu(Y_1,\cdots,Y_m) \right|. 
\end{align*}
 For simplicity of the exposition of the theorem, we assume that the locations $\mu_{a\ell}$ depend on~$n$, but the scales $\sigma_{a\ell}$ are constant. We denote $\Delta_{\mu,\infty} 
 = \max_{a \ne b \in [k] } \| \mu_a - \mu_b\|_{\infty}$ the maximum distance between the cluster centres. The following theorem establishes bounds on the recovery of Laplace mixture models. The proof is given in Appendix~\ref{appendix:proofLaplaceMixture}. 

\begin{theorem}
\label{thm:LaplaceMixtureRecovery}
Let $X_1, \cdots, X_n$ be generated from a Laplace mixture model as defined in~\eqref{eq:location_scale_mixture}. Suppose that $ k \log^2 (dk) = o( n )$ and $\min_{a \in [k]} \sum_{i \in [n]} \1\{ z_i^* = a \} \ge \alpha n k^{-1}$ for some $\alpha > 0$ (independent of $n$). 
Assume that $\Delta_{\mu,\infty} = O\left( d^{-1} \Chernoff(\cF) \right)$ and $\Chernoff(\cF) = \omega( d \sqrt{k} ( 1 + \frac{ \Delta_{\mu,\infty} }{ \sqrt{ nk^{-1} } }) )$. 
Let $\hz^{(t)}$ be the output of Algorithm~\ref{algo:iterativeClustering_parametricMixture} after $t$ steps, and suppose that the initialization verifies $\loss(z^*, \hz^{(0)}) = o\left( n k^{-1} \Delta_{\mu,\infty}^{-1} \right)$. Then, with probability of at least $1-o(1)$, it holds
 \begin{align*}
   n^{-1} \loss\left( z^*, \hz^{(t)} \right) \wle e^{ -(1+o(1)) \Chernoff(\cF)} \quad \forall t \ge \left \lfloor c \log n \right \rfloor,
 \end{align*}
 for any arbitrary constant $c > 0$. 
\end{theorem}

While we adopt a similar error decomposition approach as in prior works on clustering sub-Gaussian mixtures~\citep{chen2021optimal,gao2022iterative}, our analysis of the individual error terms is different due to the
sub-exponential nature of the data. This is done in Appendix~\ref{appendix:subExponentialEstimation}.

 The conditions $\Delta_{\mu,\infty} = O\left( d^{-1} \Chernoff(\cF) \right)$ and $\loss(z^*, \hz^{(0)}) = o\left( n k^{-1} \Delta_{\mu,\infty}^{-1} \right)$ in Theorem~\ref{thm:LaplaceMixtureRecovery} impose that the quantity $\Delta_{\mu,\infty}$ should not be too large. This might seem counter-intuitive at first. In fact, $\Delta_{\mu,\infty}$ is the \textit{maximum} distance between the cluster centres, and therefore a large $\Delta_{\mu,\infty}$ should not impact the difficulty of recovery. But the first step of Algorithm~\ref{algo:iterativeClustering_parametricMixture} estimates the quantities $\hmu_a^{(1)}$ by taking a sample mean based on the initial prediction $\hz^{(0)}$. Because (i) the sample mean is not a robust estimator and (ii) mistakes made by the initial clustering are arbitrary, those mistakes may have an enormous impact on the mean estimation if $\Delta_{\mu,\infty}$ is arbitrarily large. 
 Similar conditions, albeit usually involving $\Delta_{\mu,2} = \max_{a\ne b} \|\mu_a-\mu_b\|_2$, already appear in the study of $k$-means algorithm for (sub)gaussian mixture models. We refer the reader to ~\cite[Section~A.5]{lu2016statistical} for a counter-example showing that such a condition is necessary when studying worst-case scenarios. 
 \cite{gao2022iterative} avoid such an extra condition on $\Delta_{\mu,2}$, but at the expense of using a different loss function:  
 $
   \wloss \left( z^*, \hz \right) \weq  \sum_{a=1}^k \sum_{ \substack{ b = 1 \\ b \ne a } }^k \| \mu_a - \mu_b\|^2 \1\{ z_i^* = a, \hz_i = b \}. 
 $
 This new loss function imposes a heavier penalty on mistakes made between clusters having a large $\|\mu_a - \mu_b\|^2$. Assuming that $\wloss \left( z^*, \hz^{(0)} \right) = o( n k^{-1} \min_{a \ne b} \| \mu_a-\mu_b\|^2)$,~\cite{gao2022iterative} establish that Lloyd's algorithm attains the optimal error rate in isotropic Gaussian mixture models. The assumption of $\wloss \left( z^*, \hz^{(0)} \right)$ is stronger than the assumption on $\loss\left( z^*, \hz^{(0)} \right)$, as the former automatically rules out settings in which too many mistakes are made across cluster pairs that have a large $\|\mu_a-\mu_b\|$\footnote{More precisely, we notice that $\min_{a \ne b} \| \mu_a-\mu_b\|^2 \loss(z^*,z) \le \wloss(z^*,z)$ for all $z\in [k]^n$. Therefore, the condition $\wloss \left( z^*, \hz^{(0)} \right) = o( n k^{-1} \min_{a \ne b} \| \mu_a-\mu_b\|^2)$ implies $\loss \left( z^*, \hz^{(0)} \right) = o( n k^{-1})$, but the converse does not hold.}. 
 
 Finally, we note that in previous literature, the difficulty of clustering is expressed by a small signal-to-noise ratio, instead of a small Chernoff information. In many cases, the two are related, but as we saw in the introduction, the signal-to-noise ratio might take a different expression depending on the model considered. This is also the case for the Laplace mixture model. For example, if each dimension has a unique scale across the $k$ clusters (\textit{i.e.,} $\sigma_{1\ell} = \cdots = \sigma_{k\ell} = \sigma_{\ell}$), we have (see detailed computations in Appendix~\ref{appendix:detailed_exampleLaplace})
 \begin{align*}
    \Chernoff(\cF) \weq (1+o(1)) \min_{a \ne b} \sum_{\ell=1}^d \frac{ | \mu_{a\ell} - \mu_{b\ell} | }{ \sigma_{\ell} }. 
 \end{align*}
 This can be rewritten as 
  \begin{align*}
    \Chernoff(\cF) \weq (1+o(1)) \min_{a \ne b} \| \Sigma^{-1} (\mu_a - \mu_b) \|_1, 
 \end{align*}
 where $\Sigma$ is the diagonal matrix whose elements are $\sigma_1, \cdots, \sigma_{\ell}$. This quantity $\| \Sigma^{-1} (\mu_a - \mu_b) \|_1$ can be interpreted as an SNR. If we further restrict $\sigma_1 = \cdots = \sigma_{\ell} = \sigma$ (isotropic Laplace mixture model), we obtain 
$$
\Chernoff(\cF) \weq (1+o(1)) \frac{ \min_{a \ne b} \| \mu_a - \mu_b \|_1 }{ \sigma }. 
$$ 
For this isotropic Laplace mixture model, the error rate involves the $\ell^1$ distance, instead of the more traditional $\ell^2$ distance used in the Gaussian mixture model (see~\eqref{eq:def_snr}). 

\subsection{Bregman hard clustering of exponential family mixture models}
\label{subsection:exponentialFamilyMixtures}

 A set of pdf $\cP_{\psi}(\Theta) = \{ p_{ \theta },  \theta \in \Theta \}$ form an exponential family if each pdf $p_{\theta}$ (defined with respect to a common reference measure $\nu$) can be expressed as 
 \begin{align}
  \label{eq:def_exponential_family}
     p_{\psi,\theta} (y) \weq h(y) e^{ <u(y), \theta > - \psi(\theta) },
 \end{align}
 where $h(\cdot)$ is the carrier measure, $u(\cdot)$ is the sufficient statistics, $\psi(\theta) = \log \int h(y) e^{ <u(y), \theta >} d \nu(y)$ is the log-normalizer (also called the cumulant function), and $\theta$ is the natural parameter belonging to the space $\Theta = \{ \theta \in \R^p \colon \psi(\theta) < \infty \}$. 
 We assume that $\Theta$ is open so that $\cP(\Theta)$ forms a regular exponential family, and that $u$ is a minimal sufficient statistics\footnote{A sufficient statistic $u$ is said to be minimal if for any other sufficient statistic $\tu$, there exists a measurable function $\varphi$ such that $u = \varphi(\tu)$.}.
 Among important properties of regular exponential families, we recall that $\psi$ is a differentiable and strongly convex function which verifies $\E_{Y \sim p_{\psi,\theta}} \left[ u(Y) \right] = \nabla \psi(\theta)$~\cite[Sections~4.1 and~4.2]{banerjee2005clustering}. 
  
 We consider a family of $k$ pdf $\cF = \{ f_{\theta_1}, \cdots, f_{\theta_k} \}$ belonging to the \textit{same} exponential family, such as each $f_{\theta_a}$ can be written as  
 \begin{align}
 \label{eq:parametricMixture_exponentialFamily}
   f_{\theta_a}(x_1, \cdots, x_d) \weq \prod_{\ell=1}^d h_{\ell}(x_\ell) e^{ <u_\ell(x_\ell), \theta_{a\ell} > - \psi_\ell(\theta_{a\ell}) }. 
 \end{align}
 In other words, each coordinate $X_{i\ell}$ of $X_i$ is sampled from an exponential family with sufficient statistics~$u_{\ell}$, cumulant function $\psi_{\ell}$, and parameter $\theta_{z_i^* \ell}$. We note that, for each coordinate $\ell$, different clusters share the \textit{same} the sufficient statistic $u_{\ell}$ and cumulant function $\psi_{\ell}$, but have \textit{different} parameters $\theta_{1\ell}, \cdots, \theta_{k\ell}$. 
 Moreover, we assume that $u_\ell$ is a function from $\R$ to $\R$, but our results extend naturally if $u_\ell \colon \R \to \R^p$. 

For any convex, differentiable function $\varphi \colon \Theta \to \R$, we define its \new{Legendre transform} as $\varphi^*(y) = \sup_{\theta \in \Theta} \{ <\theta, y> - \varphi(\theta) \}$. The \new{Bregman divergence} $\dbreg_{\varphi}(\cdot \| \cdot)$ with generator $\varphi$ is defined by
 \begin{align*}
  \dbreg_{\varphi} ( x \| y ) \weq \varphi(x) - \varphi(y) - (x-y)^T \nabla \varphi(y). 
 \end{align*}
 The pdf $p_{\psi,\theta}$ defined in~\eqref{eq:def_exponential_family} can be rewritten as~\cite[Theorem~4]{banerjee2005clustering}
 \begin{align*}
     p_{\psi,\theta}(y) \weq b_{\psi}(y) e^{- \dbreg_{\psi^*} (x, \mu) },
 \end{align*}
 where $b_{\psi}(\cdot)$ is independent of $\theta$. 
Therefore, for any $f_{\theta_a} \in \cP(\Theta^d)$ we have 
 \begin{align*}
     f_{\theta_a}(x) \weq b(x) e^{ - \sum_{\ell=1}^L \dbreg_{\psi^*_{\ell}} ( u_\ell(x_{\ell}), \mu_{a\ell} ) },
 \end{align*}
 where $\mu_{a\ell} = \E_{X \sim f_{\theta_a} } \left[ u_\ell(X_\ell) \right] = \nabla \psi_{\ell}(\theta_{a\ell})$. 
Therefore, for a mixture model for the parametric family~\eqref{eq:def_exponential_family},  we can reformulate Algorithm~\ref{algo:iterativeClustering_parametricMixture} as Algorithm~\ref{algo:bregmanHardClustering}. As in Section~\ref{subsection:LaplaceMixtures}, we also define $\Delta_{\mu, \infty} = \max_{1 \le a \ne b \le k} \| \mu_{a} - \mu_{b} \|_{\infty}$. 

 \begin{algorithm}[!ht]
 \caption{Bregman hard clustering~\citep{banerjee2005clustering}}
 \label{algo:bregmanHardClustering}
 \KwInput{ Set of $n$ data points $(X_1, \cdots, X_n) \in \R^{n \times d}$, sufficient statistics $u_1, \cdots, u_d \colon \R \to \R$, convex functions $\psi_1^*, \cdots, \psi_p^* \colon \R \to \R$, number of clusters $k$, number of iteration $t_{\max}$, initial clustering $\hz^{(0)} \in [k]^n$.}
 \KwOutput{Predicted clusters $\hz \in [k]^n$.}
 \textbf{For} $t = 1 \cdots t_{\max}$ \textbf{do}
   \begin{enumerate}
    \item For $a= 1, \cdots, k$ and $\ell = 1, \cdots, d$, let $\hmu_{a\ell}  \weq \frac{ \sum_{i} \1 \left\{ \hz_i^{(t-1)} = a \right\} u_{\ell}( X_{i\ell} ) }{ \sum_{i} \1\left\{ \hz_i^{(t-1)} = a \right\} }$;
    \item For $i=1, \cdots, n$ let $\hz_i^{(t)} =  \argmin\limits_{ a \in [k]} \sum_{\ell=1}^L \dbreg_{\psi^*_{\ell}} ( u_\ell(x_{\ell}), \hmu_{a\ell} )$. 
  \end{enumerate}
 \KwReturn{$\hz = \hz^{(t_{\max})}$.} 
\end{algorithm}

 The following theorem, whose proof is provided in Appendix~\ref{appendix:proofexpofamilyMixture}, shows that Algorithm~\ref{algo:bregmanHardClustering} is rate-optimal if correctly initialised.  

\begin{theorem}
\label{thm:expofamilyMixtureRecovery} 
 Let $X_1, \cdots, X_n$ be generated from a mixture model of exponential family as defined in~\eqref{eq:parametricMixture_exponentialFamily}, and such that $u_\ell(X_{i\ell})$ is sub-exponential. 
 Suppose that $k \log^2(dk) = o(n)$ and $\min_{a \in [k]} \sum_{i\in[n]} \1\{z_i^* = a \} \ge \alpha nk^{-1}$ for some constant $\alpha > 0$. 
 Suppose that $\nabla^2 \psi^*(\mu_{a\ell}) = \Theta(1)$, $\Delta_{\mu,\infty} = O(d^{-1} \Chernoff(\cF))$ and $\Chernoff(\cF) = \omega( d \sqrt{k} ( 1 + \frac{ \Delta_{\mu,\infty} }{ \sqrt{ nk^{-1} } }) )$. 
 Let $\hz^{(t)}$ be the output of Algorithm~\ref{algo:bregmanHardClustering} after~$t$ steps, where the initialisation verifies $\loss(z^*, \hz^{(0)}) = o\left( n k^{-1} \Delta_{\mu, \infty}^{-1} \right)$. Then, with a probability of at least $1-o(1)$, it holds
 \begin{align*}
    n^{-1} \loss\left( z^*, \hz^{(t)} \right) \wle e^{ -(1+o(1)) \Chernoff(\cF)} \quad \forall t \ge \left \lfloor c \log n \right \rfloor,
 \end{align*}
 for any arbitrary constant $c > 0$. 
\end{theorem}

 We need the technical condition $\nabla^2 \psi^*_{\ell}(\mu_{a\ell}) = \Theta(1)$ to control the term $\dbreg_{\psi^*_{\ell}}( \mu_{a\ell}, \hmu_{a\ell})$ when $\hmu_{a\ell}$ is an estimate of $\mu_{a\ell}$. This condition is verified in many models of interest (such as Poisson, Negative Binomial, Exponential, or Gaussian mixture models). For example, for Poisson distributions, we have $\psi^*(x) = x \log x - 1$ and hence $\nabla^2 \psi^*(x) = x^{-1}$ is a $\Theta(1)$ if we assume that the intensities of the Poisson pdf forming the mixture are all lower-bounded. 

 The assumption that $u_{\ell}(X_{i\ell})$ is sub-exponential can be verified even if $X_{i\ell}$ has a heavier tail than exponential. For example, if $X_{i\ell}$ is log-normal, then $u_{\ell}(X_{i\ell}) = \log (X_{i\ell})$ is Gaussian and hence has sub-exponential tails. Pareto distribution provides another interesting example: if $X_{i\ell}$ is Pareto distributed with shape $\alpha$ and scale $x_m = 1$, then $\log X$ is exponentially distributed with mean~$\alpha^{-1}$. 

 Finally, we notice that, except for particular cases (such as Gaussian mixture models), the quantity $\Chernoff(\cF)$ does not have a nice closed-form expression, and we can not easily define an SNR in those models. An important example of such a quantity already appearing in the literature is the \textit{Chernoff-Hellinger divergence}, originally defined in Stochastic Block Models~\citep{abbe2015community,dreveton2023exact}, and appearing in the study of Poisson mixture models, as shown in the following example. 

 \begin{example}[Poisson mixture model]
   Consider the family $\cF = \{ f_{\theta_1}, \cdots, f_{\theta_k} \}$ of multi-variate Poisson distributions, defined by $f_{\theta_a}(x) = \prod_{\ell=1}^d \frac{ \theta_{a\ell}^{x_{\ell}} }{ x_\ell ! } e^{-\theta_{a\ell} } $ for $x \in \Z_+^d$ and $\theta_a \in \R_+^d$. Then,  
   $
    \Chernoff(\cF) \weq \min\limits_{ 1 \le a \ne b \le k } \sup\limits_{ t \in (0,1) } \sum_{ \ell =1 }^d \left( t \theta_{a\ell} + (1-t) \theta_{b\ell} - \theta_{a\ell}^t \theta_{b\ell}^{1-t} \right).
   $
 \end{example}

\section{Discussion}
\label{sec:discussion}

 \subsection{Initialisation}

 In the literature, initialisation is commonly accomplished through \new{spectral methods}, an umbrella term denoting a dimension reduction via spectral decomposition followed by clustering. Here, we perform the dimension reduction through the Singular Value Decomposition (SVD) of a well-chosen matrix~$Y$, and the clustering is done by finding an $(1+\epsilon)$-approximation of a $k$-means problem. 
 \begin{enumerate}
   \item Let $Y = \sum_{\ell=1}^{p \wedge n} s_{\ell} v_{\ell} w_{\ell}^T$ with $s_1 \ge s_2 \ge \dots \ge s_{p \wedge n}$ be the SVD decomposition of $Y \in \R^{p \times n}$. Let $V = [v_1, \cdots, v_k] \in \R^{p \times k}$ and define $\hat{M} = V Y \in \R^{k \times n}$.
   \item Return $\hz^{(0)}$, an $(1+\epsilon)$ approximation of 
   $
    \argmin\limits_{ \substack { z \in k^n \\ \check{\mu} \in \R^{d \times k} } } \sum_{i=1}^n \| \hat{M}_{\cdot i} - \check{\mu}_{z_i} \|_2^2, 
   $ where $\hat{M}_{\cdot i}$ is the $i$-th column of $\hat{M}$~\citep{kumar2004simple}. 
 \end{enumerate}
 For the Laplace mixture model, we apply the SVD directly on $Y = X$, while for an exponential family mixture, we apply it on the matrix $Y$ obtained such that $Y_{i\ell} = u_{\ell}(X_{i\ell})$ for all $i \in [n], \ell \in [d]$. The following lemma ensures that the error made by this initialisation is $o(nk^{-1} \Delta_{\mu,\infty}^{-1})$, as required by Theorems~\ref{thm:LaplaceMixtureRecovery} and~\ref{thm:expofamilyMixtureRecovery}. 
 
 \begin{lemma}
 \label{lemma:consistency_initialisation}
  Define $\delta_{\mu, 2} = \min_{1 \le a \ne b \le k} \|\mu_a - \mu_b\|_2$. Let $\hz^{(0)}$ be the clustering obtained by the initialisation described above, with $\epsilon$ being defined in step 2. 
  Assume $\min_{a \in [k]} \sum_{i \in [n] } \1\{z_i^* = a \} \ge \alpha n k^{-1}$ for some constant $\alpha > 0$ and $\Delta_{\mu,\infty} = o\left( \frac{ \delta_{\mu,2}^2 }{ (1+\epsilon) k^2 (1+ \frac{d}{n}) } \right)$. Then 
  \[
   \loss\left( z^*, \hz^{(0)} \right) \weq o\left( n k^{-1} \Delta_{\mu,\infty}^{-1}  \right). 
  \]
 \end{lemma}
 
 The proof of Lemma~\ref{lemma:consistency_initialisation} follows the same steps as in the proof of~\cite[Proposition~4.1]{gao2022iterative}, the only modification being a different choice of the loss function. The central argument in the proof of~\cite[Proposition~4.1]{gao2022iterative} is that $\| Y - \E Y \|_2 = O(\sqrt{n+d})$ with probability at least $1 - e^{- C n }$ for some $C>0$ when $Y$ has independent Gaussian entries. In our setting, $Y$ is a random matrix with independent, sub-exponential random entries, and hence its concentrate (see for example~\cite[Corollary~3.5]{bandeira2016sharp} and~\citep{dai2022tail}). 

  Finally, Lemma~\ref{lemma:consistency_initialisation} requires an additional assumption on $\delta_{\mu,2}$. While we might be able to get rid of this extra technical condition, we also notice that this condition is verified in interesting regimes. We refer the reader to the detailed example of the Laplace mixture in Section~\ref{appendix:detailed_exampleLaplace}, for which $\delta_{\mu,2}^2 = \Theta( d \Delta_{\mu,\infty}^2)$ and $ \Chernoff(\cF) = \Theta( d \Delta_{\mu,\infty} )$. In this regime, the extra condition in Lemma~\ref{lemma:consistency_initialisation} becomes $\Chernoff(\cF) = \omega( (1+\epsilon) k^2 (1+dn^{-1}))$, which is weak if $d = o(n)$.

 \subsection{Discussion and future work}

 There has been a recent surge in interest in establishing the error rates of various clustering algorithms in (sub-Gaussian) mixture models. In this section, we provide a concise overview of some of the latest and most pertinent works in this area, as well as directions for future work. 

\paragraph{Robustness to model specification, perturbed samples, and heavier tails}

 Due to its simplicity and inclusion in popular libraries like \textit{scikit-learn}, the standard Lloyd's algorithm often serves as the default choice for clustering tasks. While its optimality has been demonstrated for clustering isotropic Gaussian mixture models, its performance on other mixture models has not been studied. 
 More generally, Theorems~\ref{thm:LaplaceMixtureRecovery} and~\ref{thm:expofamilyMixtureRecovery} demonstrate that iterative algorithms are rate-optimal when the parametric family underlying the mixture distributions is known. 
 But what happens under model misspecification? 
 For instance, what error rate can we expect to achieve if we cluster a mixture of negative binomial distributions using the Bregman divergence associated with the Poisson distribution? 
 As a first result in this direction, \cite[Theorem~1]{jana2023adversarially} establishes that employing the~$\ell^1$ distance instead of the squared $\ell^2$ distance for clustering a mixture of isotropic Gaussian yields an error rate of \textit{at least} $\exp( - (2+C)^{-1} \snr_{\mathrm{isotropic}}^2 )$, where $C >0$, which is larger than the optimal rate of $\exp( - 2^{-1} \snr_{\mathrm{ isotropic}}^2)$.  

 Another important type of robustness lies in the observation of perturbed samples. Suppose that $\{X_i\}_{i\in[n]}$ is generated from a mixture model, but we observe a perturbed sample $\{\tX_i\}_{i\in[n]}$ with $\tX_i = X_i + e_i$, where the noise terms $\{e_i\}_{i\in[n]}$ verify $\|e_i\| \le \epsilon$. For sub-Gaussian $X_i$, \cite[Theorem~4.1]{patel2023consistency} establish that the mis-clustering rate of Lloyd’s algorithm on this model is \textit{at most} $\exp\left( - 4^{-1} \snr_{\mathrm{isotropic}}^2 \min \{ 1, 2\epsilon \} \right)$. 

 Because the mean is notoriously non-robust to outliers~\citep{tukey1960survey,huber1964robust}, another strategy to ensure the robustness of the iterative method is to estimate the cluster means by a robust location estimator, such as the coordinate-wise median~\citep{jana2023adversarially}, the geometric median~\citep{godichon2024robust}, or trimmed estimators~\citep{cuesta1997trimmed,garciaEscudero2008general,brecheteau2021robust}. Furthermore, robust estimators might become necessary for handling distributions with tails heavier than sub-exponential.

    

 \paragraph{High dimension regime}
 When $d$ can grow arbitrarily large,~\cite{ndaoud2022sharp} showed that the optimal error rate for clustering mixture of isotropic Gaussians with $k=2$ clusters is no longer $\exp({-2^{-1} \snr^2_{\mathrm{isotropic}} } )$ but becomes 
$
 \exp\left({ \Theta \left( \frac{ \snr^4_{\mathrm{isotropic}} }{ \snr^2_{\mathrm{isotropic}} + d n^{-1} } \right) } \right).
$ 
An extension to $k = \Theta(1)$ clusters is studied in~\cite{chen2021cutoff}. The theoretical analysis of both of these works heavily relies on the Gaussian assumption, and it remains open to extend such results to other mixture models. The key challenge is that in a mixture of two isotropic Gaussians $\frac12 \mathcal{N}(\mu_1, I_d) + \frac12 \mathcal{N}(\mu_2, I_d)$ where $d \gg n$, the dimension of the parameters of the distributions $(\mu_1, \mu_2 \in \R^d)$ is much larger than the number of data points $n$. This creates a discrepancy between the minimax error rates of algorithms with and without access to the true centres $\mu_1, \mu_2$ \citep{ndaoud2022sharp}. Exploring this phenomenon for models beyond the mixture of two isotropic Gaussians is a crucial avenue for future research.

 \paragraph{(Semi)-supervised extensions}
 Once the unsupervised error rate of various mixture models is well understood, researchers can also examine the supervised error rate of classification~\citep{li2017minimax,minsker2021minimax}. An intriguing perspective emerges when extending these analyses to a semi-supervised setting, aiming to ascertain whether a small amount of labelled data can notably diminish the clustering error rate~\citep{lelarge2019asymptotic,tifrea2023can}.

%% file: text_appendix.tex
\section{Proof of the lower-bound}
\label{section:proof_lowerBound}

\subsection{Proof of Lemma~\ref{lemma:hypothesis_testing}}
\label{proof:lemma:hypothesis_testing}

We recall that, given two pdfs  $f$ and $g$ with respect to a reference dominating measure $\nu$, the \Renyi divergence of order $t$ between $f$ and $g$ is 
\begin{align*}
 \dren_t(f \| g) \weq \frac{1}{t-1} \log \int f^{t}(x) g^{1-t}(x) d\nu(x). 
\end{align*}
 Chernoff information and \Renyi divergences are linked by the following relationship 
\begin{align*}
 \Chernoff(f,g) \weq \sup_{t \in (0,1)} (1-t)  \dren_t(f \| g) .
\end{align*}
The \Renyi divergence is not symmetric in $f$ and $g$ (except for $t=2^{-1}$), but the Chernoff information is symmetric. 

\begin{proof}[Proof of Lemma~\ref{lemma:hypothesis_testing}] 
Let $\ell(Y) = \log \frac{g_m}{f_m}(Y)$. By the definition of $\phi^{\MLE}$ and of the worst-case risk $r(\cdot)$, we have 
\begin{align*}
 r\left( \phi^{\MLE} \right) 
 \weq \max \left\{ \pr_{Y \sim f_m} \left( \ell(Y) > 0 \right), \pr_{Y \sim g_m} \left( \ell(Y) < 0 \right) \right\}.
\end{align*}
In the following, we establish upper and lower bounds for $\pr_{Y \sim f_m} \left( \ell(Y) > 0 \right)$. A similar reasoning provides bounds for $\pr_{Y \sim g_m} \left( \ell(Y) < 0 \right)$. 

\textbf{(i) Upper-bound.} 
 Applying Chernoff bounds, it holds for any $t \in (0,1)$
 \begin{align*}
   \pr_{Y \sim f_m} \left( \ell(Y) > 0 \right) 
   \weq  \pr_{Y \sim f_m} \left( e^{t \ell(Y) } > 1 \right) 
   \wle \E_{f_m} \left[ e^{ t \ell(Y) } \right].
 \end{align*}
 By the definition of $\ell(Y)$, we have 
 $\E_{f_m} \left[ e^{ t \ell(Y) } \right] \weq \E_{f_m} \left[ \left( \frac{g_m}{f_m}(Y) \right)^t\right]$. 
 By the definition of the \Renyi divergence, we also have $\E_{f_m} \left[ \left( \frac{g_m}{f_m}(Y) \right)^t \right] = e^{-(1-t) \dren_t( g_m \| f_m ) }$. Hence, 
\begin{align*}
 \pr_{Y \sim f_m} \left( \ell(Y) > 0 \right) 
 & \wle \inf_{t \in (0,1)}  e^{ - (1-t) \dren_t (g_m \| f_m ) } \\ 
 & \weq e^{- \sup_{t \in (0,1)} (1-t) \dren_t (g_m \| f_m ) } \\
 & \weq e^{- \Chernoff(f_m, g_m)}.  
\end{align*}
 We can similarly establish that 
 $ \pr_{Y \sim g_m} \left( \ell(Y) < 0 \right)  \le e^{- \Chernoff(f_m, g_m)}$, 
 and thus
 \begin{align*}
    \log r(\phi^{\MLE}) \wle - \Chernoff (f_m, g_m ).
 \end{align*}

\textbf{(ii) Lower-bound.}
  For any $s \ge 0$ and $t \in (0,1)$, we have  
\begin{align*}
 \pr_{f_m} \left( \ell(Y) > 0 \right) 
 & \wge \E_{f_m} \left[ e^{ - t \ell(Y)} e^{t \ell(Y) } \1 \{ 0 \le \ell(Y) \le s \} \right] \\
 & \wge e^{-ts} \ \E_{f_m} \left[ e^{ t \ell(Y)} \1\{ 0 \le \ell(Y) \le s\} \right]. 
\end{align*}
Next, we define $h_{1-t} = \frac{ f_m^{1-t} g_m^t }{ \int f_m^{1-t} g_m^t }$. 
We notice that $\int f_m^{1-t} g_m^t = e^{- (1-t) \dren_t( g_m \| f_m ) }$, and furthermore 
\begin{align*}
 \E_{f_m} \left[ e^{ t \ell(Y)} \1\{ 0 \le \ell(Y) \le s \} \right] 
 & \weq \int f_m^{1-t}(y) g_m^t(y) \1\{0 \le \ell(y) \le s \} \, d \nu(y) \\
 & \weq e^{ -(1-t)\dren_t(g_m \| f_m) } \pr_{h_{1-t}}(\ell(\tY) \in [0,s] ), 
\end{align*}
where $\tY$ is a random variable distributed from $h_{1-t}$. 
Therefore,
\begin{align*}
 \pr_{f_m} \left( \ell(Y) > 0 \right) 
 & \wge e^{-ts} e^{ -(1-t)\dren_t( g_m \| f_m) } \pr_{h_{1-t}}(\ell(\tY) \in [0,s] )  \\
 & \wge  e^{-s} e^{ -\Chernoff( f_m, g_m) } \pr_{h_{1-t}}(\ell(\tY) \in [0,s] ),
\end{align*}
where we used $e^{-ts} \ge e^{-s}$ and $ (1-t)\dren_t( g_m \| f_m) \le \Chernoff(f_m \| g_m) $ because $t \in (0,1)$. Since the previous inequality is valid for any $t$, we obtain by taking $t= \frac12$ that
\begin{align*}
 \pr_{f_m} \left( \ell(Y) > 0 \right) 
 & \wge  e^{-s} e^{ -\Chernoff( f_m, g_m) } \pr_{h_{1/2}}(\ell(\tY) \in [0,s] ).
\end{align*}
Next, we notice that $\pr_{h_{1/2}}(\ell(\tY) \in [0,s] ) = \pr_{h_{1/2}}(\ell(\tY) \in [-s,0] )$. Let $s = \sqrt{2 \E_{h_{1/2}} \left[ \ell(\tY)^2 \right]}$. Then, Chebyshev's inequality implies that $\pr_{h_{1/2}}( \abs{\ell(\tY)} > s) \wle s^{-2} \E_{h_{1/2}} \left[ \ell(\tY)^2 \right] \wle \frac12$. 
\end{proof}

\subsection{Proof of Theorem~\ref{thm:lower_bound_error}}

This section is devoted to the proof of Theorem~\ref{thm:lower_bound_error}, which provides a lower bound of the minimax loss $\inf_{\hz} \sup_{ z \in [k]^n } \loss(z,\hz)$.
Without loss of generality, we assume that the minimum in the definition of $\Chernoff(\cF)$ in~\eqref{eq:def_chernoff_family} is achieved at $a = 1$ and $b=2$, that is 
\[
\Chernoff(f_1, f_2) \weq \min_{1 \le a \ne b \le k} \Chernoff(f_a, f_b).
\]
For any cluster membership vector $z \in [k]^n$, we denote for all $a \in [k]$ the set of all indices $i \in [n]$ belonging to cluster $a$ by 
\[
 \Gamma_a(z) = \{i \in [n] \colon z_i=a \}.
\]

Recall from Definition~\eqref{eq:def_hamming_loss} that $\loss(z^*,\hz) = \min_{\tau \in \Sym(k)} \ham( z^*, \tau \circ \hz )$. In particular, the loss function involves a minimum over all permutations $\tau \in \Sym(k)$, making it hard to study directly. 
But, we can get rid of this minimum in the definition of the loss if we deal with vectors having a loss small enough, because in that case the $\min$ is attained by a unique minimiser. 
We state without proof the following lemma from~\cite{avrachenkov2020community}. 

\begin{lemma}[Lemma~C.5 in~\cite{avrachenkov2020community}] 
 \label{lemma:unique_permutation_minimiser}
 Let $z_1, z_2 \in [k]^n$ such that $\ham(z_1, \tau^* \circ z_2) < \frac12 \min_{a \in [k]} |\Gamma_a(z_1)|$ for some $\tau^* \in \Sym(k)$. Then $\tau^*$ is the unique minimiser of $\tau \in \Sym(k) \mapsto \ham(z_1, \tau \circ z_2)$. 
\end{lemma}

 Following the same proof strategy as previous works on clustering block models~\citep{gao2018community,chen2021optimal}, we define a clustering problem over a subset of $[k]^n$ to avoid the issues of label permutations. 
 Let $\alpha > 0$ be an arbitrary constant independent of $n$. We define $ \cZ = \cZ_{n,k} \subset [k]^n$ the set of vectors such that all clusters have size at least $\alpha n k^{-1}$:
\[
 \cZ \weq \{ z \in [k]^n \colon \ |\Gamma_a(z)| \wge \alpha n k^{-1} \quad \text{ for all } a \in [k] \}. 
\]
Let $z^* \in \cZ$. For every cluster $a \in [k]$, 
collect the indices of the $|\Gamma_a(z^*)| - \frac{\alpha n}{5k}$ smallest indices~$i$'s in $\Gamma_a(z^*) = \{ i \in [n] \colon z^*_i = a \}$ into a set $T_a$. 
Let $T = T_1 \cup T_2 \cup_{ a = 3}^k \Gamma_a(z^*)$ and define a new parameter space $\tcZ$ 
\[
 \tcZ \weq \{ z \in \cZ \colon z_i = z^*_i \text{ for all } i \in T \quad \text{ and } \quad z_i \in \{1,2\} \text{ if } i \in T^c \}. 
\]
Because $T^c = T_1^c \cup T_2^c$, this new space $\tcZ$ is composed of all cluster labelling~$z$ that 
only differs from~$z^*$ on the indices $i$'s that do not belong to $T_1$ or $T_2$. 
By construction of $\tcZ$, we have for any $z, z' \in \tcZ$  
\begin{align*}
 \ham(z,z') \weq \sum_{i=1}^n \1\{z_i \ne z'_i\} \wle \left|T^c\right| \weq 2 \, \frac{\alpha n}{5k}.
\end{align*}
Because $z \in \tcZ \subset \cZ$, we have by definition of $\cZ$ that $\min_{a \in [k]} |\Gamma_a(z)| > \alpha n k^{-1}$. Therefore, the previous inequality ensure that $\ham(z,z') < 2^{-1} \min_{a \in [k]} |\Gamma_a(z)|$ for all $z,z' \in \cZ$. We can thus apply Lemma~\ref{lemma:unique_permutation_minimiser} to establish that 
\begin{align}
\label{eq:loss_in_reduceSpace}
 \forall z, z' \in \tcZ \ \colon \quad \loss(z,z') \weq \ham(z,z') \weq \sum_{i \in T^c} \1\{ z_i \ne z'_i \}. 
\end{align}
Because $\tcZ \subset \cZ \subset [k]^n$, we also have 
\begin{align*}
 \inf_{\hz} \sup_{ z \in [k]^n } \E_z \loss(z, \hz) 
 \wge \inf_{\hz} \sup_{z \in \tcZ} \E_z \loss(z, \hz)
 \weq \inf_{\hz} \sup_{z \in \tcZ} \E_z \ham(z, \hz),
\end{align*}
where the equality follows from~\eqref{eq:loss_in_reduceSpace}. 
Bounding the minimax risk by the Bayes risk leads to
\begin{align*}
  \inf_{ \hz } \sup_{z \in \tcZ} \E_z \left[ \ham(z,\hz) \right] 
  \wge \inf_{\hz} \frac{1}{|\tcZ|} \sum_{z \in \tcZ} \E_z \left[ \ham(z,\hz) \right]. 
\end{align*}
Moreover,
\begin{align*}
 \inf_{\hz} \frac{1}{|\tcZ|} \sum_{z \in \tcZ} \E_z \ham(z,\hz)
 & \weq \inf_{\hz_1, \cdots, \hz_n} \frac{1}{|\tcZ|} \sum_{z \in \tcZ} \sum_{i \in T^c} \pr_z( \hz_i \ne z_i ) \\
 & \weq \sum_{i \in T^c} \inf_{\hz_i} \frac{1}{|\tcZ|} \sum_{z \in \tcZ} \pr_z( \hz_i \ne z_i ).
\end{align*}
Therefore, we can conclude from these previous inequalities that
\begin{align}
\label{eq:in_proof_lowerBound_getRidPermutation}
 \inf_{\hz} \sup_{ z \in [k]^n } \E_z \loss(z, \hz)  
 \wge \sum_{i \in T^c} \inf_{\hz_i} \frac{1}{|\tcZ|} \sum_{z \in \tcZ} \pr_z( \hz_i \ne z_i ).
\end{align}

Fix $i \in T^c$ and define $\tcZ_{a}^{(i)} = \{ z \in \tcZ \colon z_i = a \}$ for $a \in \{1,2\}$. We observe that $\tcZ_1^{(i)} \cup \tcZ_2^{(i)} = \tcZ$ and that $\tcZ_1^{(i)} \cap \tcZ_2^{(i)} = \emptyset$. 
Let $f \colon \tcZ_1^{(i)} \to \tcZ_2^{(i)}$ such that for any $z \in \tcZ_1^{(i)}$ we have $f(z) \in \tcZ_2^{(i)}$ defined by 
\begin{align*}
 \left( f(z) \right)_j \weq
 \begin{cases}
     z_j & \text{ if } j \ne i, \\
     2 & \text{ if } j = i.
 \end{cases}
\end{align*}
The function $f$ defines a one-to-one mapping from $\tcZ_1^{(i)}$ to $\tcZ_2^{(i)}$. Because these two sets partition $\tcZ$, we have $|\tcZ_1^{(i)}| = 2^{-1} |\tcZ|$. Moreover,  
\begin{align}
 \inf_{\hz_i} \frac{1}{|\tcZ|} \sum_{z \in \tcZ} \pr_z( \hz_i \ne z_i ) 
 & \weq \inf_{\hz_i} \frac{1}{|\tcZ|} \Big( \sum_{ z \in \tcZ_1^{(i)} } \pr_{z} \left( \hz_i \ne 1 \right) + \sum_{ z \in \tcZ_2^{(i)} } \pr_{ z } \left( \hz_i \ne 2 \right) \Big) \nonumber  \\ 
 & \weq \inf_{\hz_i} \frac{1}{|\tcZ|} \sum_{z \in \tcZ_1^{(i)} } \left( \pr_{z} \left( \hz_i \ne 1 \right) + \pr_{f(z)} \left( \hz_i \ne 2 \right) \right) \nonumber \\
 & \wge \frac{1}{|\tcZ|} \sum_{ z \in \tcZ_1^{(i)} } \inf_{\hz_i} \left( \pr_{z} \left( \hz_i \ne 1 \right) + \pr_{f(z)} \left( \hz_i \ne 2 \right) \right). 
 \label{eq:in_proof_lowerBound_individual}
\end{align}

We are now reduced to the problem of estimating $\hz_i$, and the best estimator for this task is 
\begin{align*}
    \hz_i^{\MLE} \weq
    \begin{cases}
        1 & \text{ if } f_1(X_i) > f_2(X_i), \\
        2 & \text{ otherwise.} 
    \end{cases}
\end{align*} 
In other words, we are in the setting of Lemma~\ref{lemma:hypothesis_testing}, where we observe a single sample $X_i$ and want to discriminate between $H_0 \colon X_i \sim f_1$ and $H_1 \colon X_i \sim f_2$. 
Because $\Chernoff(\cF) = \omega( \log k )$, Lemma~\ref{lemma:hypothesis_testing} ensures that 
\begin{align*}
 \inf_{\hz_i} \left( \pr_{z} \left( \hz_i \ne 1 \right) + \pr_{f(z)} \left( \hz_i \ne 2 \right) \right)
 \wge e^{-(1+o(1)) \Chernoff( f_1,f_2 )},
\end{align*}
and thus 
\begin{align*}
 \inf_{\hz_i} \frac{1}{|\tcZ|} \sum_{z \in \tcZ} \pr_z( \hz_i \ne z_i ) 
 \wge \frac{ |\tcZ_1| }{ |\tcZ| } e^{-(1+o(1)) \Chernoff( f_1, f_2 )}
 \weq  \frac{ 1 }{ 2 } e^{-(1+o(1)) \Chernoff( \cF )}.  
\end{align*}
Going back to inequality with~\eqref{eq:in_proof_lowerBound_getRidPermutation} leads to 
\begin{align*}
 \inf_{\hz} \sup_{ z \in [k]^n } \E_z \loss(z, \hz)  
 & \wge \sum_{i \in T^c} \inf_{\hz_i} \frac{1}{|\tcZ|} \sum_{z \in \tcZ} \pr_z( \hz_i \ne z_i ) \\
 & \wge \frac{ |T^c| }{2} e^{-(1+o(1)) \Chernoff( \cF )} \\
 & \weq \frac{ \alpha n }{5k} e^{-(1+o(1)) \Chernoff( \cF )},
\end{align*}
where the last line uses $|T^c| = \frac{2 \alpha n}{5 k}$. 
We finish the proof by using the assumption $\Chernoff(\cF) = \omega( \log k )$.

\section{Proofs of Section~\ref{subsection:parametricMixtureModel}}

 \subsection{Decomposition of the error}
 \label{appendix:subsection:decompositionError}
 We first notice that  
 \begin{align}
  \label{eq:in_proof_usefull_way_to_express_Hamming_eq}
  \loss\left( z^*, \hz^{(t)} \right) 
 \weq \sum_{i \in [n] } \1 \left\{ z_i^{(t)} \ne z^*_i \right\} 
  \weq \sum_{i \in [n]} \sum_{ \substack{ b \in [k] \backslash \{z^*_i\} } } \1 \left\{ z_i^{(t)} = b \right\}.
 \end{align}
Combining 
$$
\1 \left\{ z_i^{(t)} = b \right\}  = \left( \1 \left\{ z_i^{(t)} = b \right\} \right)^2 
$$ and  
\begin{align*}
 \1 \left\{ z_i^{(t)} = b \right\} 
 \weq \1 \left\{ \forall a \in [k] \backslash\{b\} \colon \hell_b^{(t)}(X_i) > \hell_a^{(t)}(X_i) \right\} 
 \wle \1 \left\{ \hell_{b}^{(t)}(X_i) > \hell_{z_i^*}^{(t)}(X_i) \right\}
\end{align*}
with~\eqref{eq:in_proof_usefull_way_to_express_Hamming_eq}, we obtain 
\begin{align}
\label{eq:in_proof_temporary_expression_loss}
 \loss\left( z^*, \hz^{(t)} \right) & \wle \sum_{i \in [n]} \sum_{ \substack{ b \in [k] \backslash \{z^*_i\} } } \1 \left\{ z_i^{(t)} = b \right\} \1 \left\{ \hell_{b}^{(t)}(X_i) > \hell_{z_i^*}^{(t)}(X_i) \right\}. 
\end{align}
Let us study the term $\1 \left\{ \hell_{b}^{(t)} (X_i) > \hell_{z_i^*}^{(t)} (X_i) \right\}$. 
For any $\delta > 0$, we have  
\begin{align*}
 & \1 \left\{ \hell_b^{(t)}(X_i) > \hell_{z_i^*}^{(t)}(X_i) \right\} \\
 \wle & \1 \left\{ \ell_{z_i^*}(X_i) - \ell_{_b}(X_i) < \delta \right\} 
 + \1 \left\{ \delta < \hell_{b}^{(t)}(X_i) - \ell_b(X_i) + \ell_{z_i^*}(X_i) - \hell_{z_i^*}^{(t)}(X_i) \right\}. 
\end{align*}
Using $\1\{ 1 \le x + y \} \le \1\{ 1 \le |x|+|y| \} \le |x| + |y|$ for any $x, y \in \R$, we can further upper-bound the two terms appearing in the right-hand side of the last inequality by  
\begin{align*}
 & \1 \left\{ \delta < \hell_{b}^{(t)}(X_i) - \ell_b(X_i) + \ell_{z_i^*}(X_i) - \hell_{z_i^*}^{(t)}(X_i) \right\} \\
 & \wle \delta^{-1} \left( \left| \hell_{b}^{(t)}(X_i) - \ell_b(X_i) \right| + \left| \ell_{z_i^*}(X_i) - \hell_{z_i^*}^{(t)}(X_i) \right| \right) \nonumber \\
 & \wle 2 \delta^{-1} \max_{a \in [k]} \left| \hell_{a}^{(t)}(X_i) - \ell_a(X_i) \right|. 
\end{align*}
 Therefore, we obtain using Expression~\eqref{eq:in_proof_temporary_expression_loss} 
\begin{align*}
 \loss\left( z^*, \hz^{(t)} \right) 
 & \wle \xi_{\ideal}(\delta) + \xi_{\excess}^{(t)}(\delta), 
\end{align*}
where 
\begin{align*}
 \xi_{\ideal}(\delta) & \weq \sum_{i \in [n]} \sum_{ \substack{ b \in [k] \backslash \{z^*_i\} } } \1 \left\{ \ell_{z_i^*}(X_i) - \ell_{_b}(X_i) < \delta \right\}, \\ 
 \xi_{\excess}^{(t)}(\delta) & \weq 2 \delta^{-1} \sum_{i \in [n]} \sum_{ \substack{ b \in [k] \backslash \{z^*_i\} } } \1 \left\{ z_i^{(t)} = b \right\} \max_{a \in [k]} \left| \hell_{a}^{(t)}(X_i) - \ell_a(X_i) \right|. 
\end{align*}


\subsection{Study of the ideal error}
\label{subsection:study_ideal_error}

\begin{proof}[Proof of Lemma~\ref{lemma:ideal_error_rate}]
Taking the expectations in (\ref{eq:def_ideal_error}), we have
\begin{align*}
 \E \xi_{\ideal}(\delta) 
 \weq \sum_{i \in [n]} \sum_{ \substack{ b \in [k] \backslash \{z^*_i\} } } \pr \left( \ell_{_b}(X_i) - \ell_{z_i^*}(X_i) > \delta \right).
\end{align*}
 Chernoff's bound (see the proof of Lemma~\ref{lemma:hypothesis_testing}) yields that  
\begin{align*}
 \pr \left( \ell_{_b}(X_i) - \ell_{z_i^*}(X_i) > \delta \right) \wle e^{ \delta -(1+o(1)) \Chernoff( f_{{z_i^*}}, f_b ) },
\end{align*}
and therefore
\begin{align*}
 \E \xi_{\ideal}(\delta) 
 \wle n k e^{ \delta -(1+o(1)) \Chernoff( \cF ) }.
\end{align*}
Now, because $\delta = o\left( \Chernoff(\cF) \right)$, Markov's inequality implies that 
\begin{align*}
 \pr \left( \xi_{\ideal}(\delta)  > \E \xi_{\ideal}(\delta) e^{ \sqrt{ \Chernoff(\cF) } } \right) \wle e^{ - \sqrt{\Chernoff(\cF) } }.
\end{align*}
Therefore with probability of at least $1 - e^{-\sqrt{\Chernoff(\cF)}}$ we have
\begin{align*}
 \xi_{\ideal}(\delta) \wle \E \xi_{\ideal}(\delta) e^{ \sqrt{ \Chernoff(\cF) } } \wle n k e^{ -(1+o(1)) \Chernoff( \cF ) }, 
\end{align*}
because $ \Chernoff( \cF ) = \omega(1)$.
\end{proof}

\subsection{Proof of Lemma~\ref{lemma:generalParametricMixtureRecovery}}
\label{appendix:proof_thm:generalParametricMixtureRecovery}

\begin{proof}[Proof of Lemma~\ref{lemma:generalParametricMixtureRecovery}]  
 Assumption $\Chernoff(\cF) = \omega( \log (k^2 \tau) )$ combined with $\loss(z^*,\hz^{(0)}) = o(nk^{-1} \tau)$ and Condition~\ref{condition:excess_error} ensures that $\loss(z^*,\hz^{(t)}) = o(nk^{-1} \tau)$ for every $t \ge 0$. Therefore, 
\begin{align*}
 n^{-1} \loss\left(z^*, z^{(t)}\right) 
 & \wle \left( \sum_{\tau=0}^t \left(\frac{c'}{1-c} \right)^\tau \right) e^{-(1+o(1)) \Chernoff(\cF)} + \left( \frac{c'}{1-c} \right)^t \\
 & \wle e^{-(1+o(1)) \Chernoff(\cF)} + \left( \frac{c'}{1-c} \right)^t,
\end{align*}
where the second inequality holds because $\sum_{\tau=0}^\infty \left(\frac{c'}{1-c} \right)^{\tau} =O(1)$ and $\Chernoff(\cF) = \omega( 1 )$. 
Because $n^{-1} \loss\left(z^*, z^{(t)}\right)$ takes value in the set $\{ j n^{-1}, j \in \{0, \cdots, n\} \}$, the term $\left( \frac{c'}{1-c} \right)^t$ is negligible if $\left( \frac{c'}{1-c} \right)^t = o(n^{-1})$, which occurs whenever $t \ge \left \lfloor \log \left( \frac{1-c}{c'} \right) \log n \right \rfloor$. 
\end{proof}

\section{Mean and scale estimations of sub-exponential random variables}
\label{appendix:subExponentialEstimation}

\subsection{Large deviations of a sum of sub-exponential random variables }

 Let $Y_1, \cdots, Y_n$ be independent, zero-mean, sub-exponential random variables. For any subset of indices $S \subseteq [n]$, Bernstein's inequality~\cite[Theorem~2.8.2]{vershynin2018high} ensures that 
 \begin{align}
   \pr \left( \left| \frac{1}{ \sqrt{|S|} } \sum_{i \in S} Y_i \right| \ge t \right) \wle 
   \begin{cases}
       2 \exp( - c t^2 ) & \text{ if } t < C \sqrt{|S|}, \\
       2 \exp( - t \sqrt{|S|} ) & \text{ if } t \ge C \sqrt{|S|}. 
   \end{cases}
   \label{eq:Bernstein_subexponential}
 \end{align}
 for some positive constants $c$ and $C$. The following lemma establishes a uniform upper bound on the quantity $\frac{1}{ \sqrt{|S|} } \sum_{i \in S} Y_i $ over all the sets $S \subset [n]$ of size smaller than $s$.   

 \begin{lemma}
\label{lemma:uniformBounding_deviations_sum_subexponential_rv}
   Let $Y_1, \cdots, Y_n$ be independent, zero-mean, sub-exponential random variables. Let $C$ be the constant in~\eqref{eq:Bernstein_subexponential}. For any $s = \omega( \log^6 n)$, we have 
   \begin{align*}
     \max_{ \substack{ S \subset [n] \\ |S| \le s } } \frac{1}{ \sqrt{ |S| } } \left| \sum_{i \in S} Y_i \right| 
     \wle C \sqrt{ s } 
   \end{align*}
   with probability at least $1 - 6 e^{- C \sqrt{s}/2 }$. 
 \end{lemma}
 

 \begin{proof}  We will use~\eqref{eq:Bernstein_subexponential} with $t = C \sqrt{ s }$. By a union bound, we have 
    \begin{align*}
     \pr \left( \max_{ \substack{ S \subset [n] \\ |S| \le s } } \frac{1}{ \sqrt{ |S| } } \left| \sum_{i \in S} Y_i \right| 
     \ge t \right) 
     \wle & \sum_{ \ell = 1 }^{ \lfloor s \rfloor } \pr \left( \max_{ S \subset [n] \colon |S| = \ell } \frac{1}{\sqrt{ |S| }} \left| \sum_{i \in S} Y_i \right| \ge t \right) \\
     \wle & \sum_{\ell = 1}^{ \lfloor s \rfloor } \binom{n}{\ell} 2\exp\left( -t \sqrt{\ell} \right). 
    \end{align*}
   For any integers $a,b$ verifying $1 \le a < b \le n$, we define 
   \begin{align*}
       S_n(a,b) \weq \sum_{\ell = a}^{ b } \binom{n}{\ell} 2\exp( -t \sqrt{\ell} ). 
   \end{align*}
   Let us chose a sequence $\beta$ verifying $\beta = \omega(\log^2 n)$ and $\beta^2 \log^2 n = o( s )$. Such a choice is possible because $s = \omega(\log^6 n)$\footnote{Since $s = \omega(\log^6 n)$, there exists a diverging sequence $\omega_n = \omega(1)$ such that $s = \omega_n \log^6 n$. We then let $\beta = \omega_n^{1/4} \log^2 n$.}. We then split the sum $S_n( 1, \lfloor s \rfloor )$ into three terms as follows:
   \begin{align*}
       S_n(1, \lfloor s \rfloor ) \weq 
        S_n\left(1, \left\lfloor \sqrt{ s / \beta } \right\rfloor \right) 
        +  S_n\left( \left\lfloor \sqrt{ s / \beta } \right\rfloor + 1, \left \lfloor \sqrt{ s \beta } \right\rfloor \right) 
        + S_n\left( \left\lfloor \sqrt{ s \beta } \right\rfloor+1, \lfloor s \rfloor \right),
   \end{align*}
   and we show that each of these three terms is less than $2 e^{- C \sqrt{s}/2 }$.
   
   For ease of notation, we drop the $\lfloor \cdot \rfloor$. We also recall the inequality $\binom{n}{\ell} \le \left( \frac{en}{\ell} \right)^\ell$, yielding $\binom{n}{\ell} \le (en)^{\sqrt{ s / \beta}}$ for all $\ell \le \sqrt{ s / \beta }$. 
   
   (i) Recalling that $t = C \sqrt{s}$, we have for all $1 \le \ell \le \sqrt{ s / \beta}$
   \begin{align*}
       \binom{n}{\ell} 2\exp( -t \sqrt{\ell} ) \wle 2e^{-C \sqrt{ s } } \left(e n \right)^{ \sqrt{ s / \beta } }, 
   \end{align*}
   and therefore 
   \begin{align*}
     S_n\left(1, \sqrt{ s / \beta } \right) 
     & \wle 2e^{-C \sqrt{s} } \sqrt{ \frac{ s }{ \beta } } \left(e n \right)^{ \sqrt{ s / \beta } } \\
     & \weq 2 e^{- C \sqrt{s} \left( 1 - \frac{\log( s / \beta ) }{ 2C \sqrt{ s } } - \frac{ \log( e n ) }{ C \sqrt{\beta} } \right) }.
   \end{align*}
   Because $ \log n = o ( \sqrt{\beta})$ (since $\beta = \omega(\log^2 n)$) and $ \log(s / \beta) = o( \sqrt{s})$ (since $\beta^2 \log^2 n = o( s )$), we have for $n$ large enough, 
   \begin{align*}
        S_n\left(1, \sqrt{s / \beta } \right)  \wle 2 e^{- \frac{C}{2} \sqrt{s}}.
   \end{align*}
   
   (ii) Similarly, 
   \begin{align*}
      S_n\left( \sqrt{ s / \beta } + 1, \sqrt{ s \beta } \right)  
      & \wle 2 e^{-C s \beta^{-1/2} } \sqrt{ s\beta } \left( en \right)^{\sqrt{s \beta} } \\
      & \weq 2 e^{ - C \frac{s}{ \sqrt{\beta} } \left( 1 - \frac{ \sqrt{\beta} \log (s \beta ) }{ 2 C s } - \frac{ \beta \log (e n) }{ C \sqrt{s} } \right) }.
   \end{align*}
   With our choice of $\beta$, we have $\sqrt{\beta} \log (s \beta ) = o(s)$ and $ \beta \log n = o(\sqrt{s}) $, and therefore for $n$ large enough, 
   \begin{align*}
       S_n\left( \sqrt{s / \beta } + 1, \sqrt{s \beta } \right) \wle 2 e^{- \frac{C}{2} \frac{s}{ \sqrt{\beta} } } \wle 2 e^{- \frac{C}{2} \sqrt{s} }, 
   \end{align*}
   because $ \sqrt{ s / \beta } = \omega(1)$. 
   
   (iii) Finally, 
   \begin{align*}
       S_n\left( \sqrt{s \beta } + 1, s \right)
       & \wle 2 e^{- C s \beta^{1/2} } s \left( e n \right)^{s} 
       \wle 2 e^{- C s \beta^{1/2} \left( 1 - \frac{\log s}{C s \sqrt{\beta} } - \frac{\log(en)}{ C \sqrt{\beta} }\right) }
       \wle 2 e^{ - \frac{C}{2} \sqrt{s} } 
   \end{align*}
   for $n$ large enough. This concludes the proof. 
 \end{proof}

\subsection{Quality of the estimates \texorpdfstring{$\hmu_{a\ell}$}{hmu} and \texorpdfstring{$\hsigma_{a\ell}$}{hsigma}}
\label{eq:quality_estimator_hmu}

In this section, we upper bound the quantity $\| \hmu_a - \mu_a \|$. For any cluster labeling $z \in [k]^n$ and cluster $a \in [k]$, let $\Gamma_a(z) = \{ i \in [n] \colon z_i = a \}$. Recall that the empirical location $\hmu_{a\ell}(z)$ and scale $\hsigma_{a\ell}(z)$ estimated from $z$ are defined by 
\begin{align*}
 \hmu_{a\ell}(z) & \weq \frac{1}{ \left| \Gamma_a(z) \right|} \sum_{i \in n} \1 \{ z_i = a \} X_{i\ell} \\
 \hsigma_{a\ell}(z) & \weq \frac{1}{ \left| \Gamma_a(z) \right|} \sum_{i \in n} \1 \{ z_i = a \} \left| X_{i\ell} - \hmu_{a\ell}(z) \right|. 
\end{align*}
Finally, let $\bmu_{a\ell}(z) = \E \hmu_{a\ell}(z)$ and $\xi_{a\ell}(z) = \hmu_{a\ell}(z) - \bmu_{a\ell}(z)$. 

\begin{proposition}
 \label{prop:quality_locationEstimator} 
 Let $\Delta_{\mu,\infty} = \max\limits_{ \substack{ a,b \in [k] \\ \ell \in [d] } } | \mu_{a\ell} - \mu_{b\ell} |$. 
 Assume $\min_{a \in [k]} |\Gamma_a(z^*) | \ge \alpha n k^{-1}$ for some constant $\alpha > 0$. 
 For $z \in [k]^n$ and some constant $C$, define 
 \begin{align*}
    \cE_{\mu}(z) \weq 
    \left\{ 
    \max_{ \substack{ a \in [k] \\ \ell \in [d] } } \, | \hmu_{a\ell}(z) - \mu_{a\ell} |  \wle  \frac{2k \Delta_{\mu,\infty} }{\alpha n} \ham(z^*, z) 
   + C' \sqrt{ \frac{ k \ham(z^*,z) }{ n } } 
   + \sqrt{ \frac{k}{n} }
   \right\} 
 \end{align*}
 with $C' = 2 \sqrt{ \frac{2}{\alpha} } \left( C + \sqrt{2} \right)$. 
 There exists a constant $C>0$ such that the event $ \bigcap\limits_{ \substack{ \ham(z^*, z) \le \frac{\alpha n}{2k} } } \cE_{\mu}(z)$ holds with probability at least $1 - k d \left( 4 e^{-\alpha \sqrt{\frac{n}{k} }} + 6 e^{- \frac{C}{2} \sqrt{ \frac{\alpha n}{2k} } } \right)$. 
\end{proposition}

\begin{proof} Let $a \in [k]$ and $\ell \in [d]$. 
 A first triangle inequality leads to 
 \begin{align}
  \label{eq:in_proof_first_triangle}
    | \hmu_{a\ell}(z) - \mu_{a\ell} | \wle | \hmu_{a\ell}(z) - \hmu_{a\ell}(z^*) | + | \hmu_{a\ell}(z^*) - \mu_{a\ell} |. 
 \end{align}
 Moreover, another triangle inequality yields that  
 \begin{align}
   | \hmu_{a\ell}(z) - \hmu_{a\ell}(z^*) | 
   & \weq | \hmu_{a\ell}(z) - \bmu_{a\ell}(z) + \bmu_{a\ell}(z) - \bmu_{a\ell}(z^*) + \bmu_{a\ell}(z^*) - \hmu_{a\ell}(z^*) | \nonumber \\ 
   & \wle | \bmu_{a\ell}(z) - \bmu_{a\ell}(z^*) | + | \xi_{a\ell}(z) - \xi_{a\ell}(z^*) |. 
   \label{eq:in_proof_second_triangle}
 \end{align}
 Therefore, combining~\eqref{eq:in_proof_first_triangle} and~\eqref{eq:in_proof_second_triangle} gives
 \begin{align}
 \label{eq:separation_error_mean_estimated}
    | \hmu_{a\ell}(z) - \mu_{a\ell} | \wle | \hmu_{a\ell}(z^*) - \mu_{a\ell} | + | \bmu_{a\ell}(z) - \bmu_{a\ell}(z^*) | + | \xi_{a\ell}(z) - \xi_{a\ell}(z^*) |. 
 \end{align}
 We will now upper--bound separately the three terms appearing on the right-hand side of~\eqref{eq:separation_error_mean_estimated}.  

 \textit{(i) Bounding $ | \hmu_{a\ell}(z^*) - \mu_{a\ell} |$.} Let $\epsilon_{i\ell} = X_{i \ell{}} - \mu_{z_i^* \ell}$. We have 
   \begin{align*}
      | \hmu_{a\ell}(z^*) - \mu_{a\ell} | 
      \weq \frac{1}{|\Gamma_a(z^*)|} \left| \sum_{i \in \Gamma_a(z^*) } \epsilon_{i\ell} \right|.
   \end{align*}
   By concentration of sub-exponential random variables (see Equation~\eqref{eq:Bernstein_subexponential}), we have 
   \begin{align*}
    \pr \left( \frac{1}{|\Gamma_a(z^*)|} \left| \sum_{i \in [n] } \1\{ z_i^* = a \} \epsilon_{i\ell} \right| \ge t \right) 
    \wle 2e^{-t |\Gamma_a(z^*)| }.
   \end{align*}
   Because $ |\Gamma_a(z^*)| \ge \alpha n / k$, this implies that with $t = \sqrt{k}$
   \begin{align}
   \label{eq:upperBound_sumOverTrueClusters}
     \frac{1}{|\Gamma_a(z^*)|} \left| \sum_{i \in [n] } \1\{ z_i^* = a \} \epsilon_{i\ell} \right| 
     \wle \sqrt{ \frac{k}{n} } 
   \end{align}
   with probability at least $1 - 2e^{-\alpha \sqrt{n/k}}$. Therefore, a union bound over $a\in[k]$ and $\ell \in [d]$ ensures that 
   \begin{align*}
     \max_{ \substack{a \in [k] \\ \ell \in [d] } } | \hmu_{a\ell}(z^*) - \mu_{a\ell} | \wle \sqrt{ \frac{k}{n} } 
   \end{align*}
 with probability at least $1 - 2 dk e^{-\alpha \sqrt{n/k}}$. 
 
   \textit{(ii) Bounding $| \bmu_{a\ell}(z) - \bmu_{a\ell}(z^*) |$.}  
   Since $\bmu_{a\ell}(z^*) = \mu_{a\ell}$ and $\sum_{i \in [n]} \sum_{b \in [k] } \1 \{ z_i=a, z_i^*=b \} = |\Gamma_a(z)|$, 
    \begin{align*}
      \bmu_{a\ell}(z) - \bmu_{a\ell}(z^*) 
      & \weq \frac{1}{\left| \Gamma_a(z) \right| } \sum_{i \in [n]} \sum_{b \in [k] } \1 \{ z_i=a, z_i^*=b \} \left( \mu_{b\ell} - \mu_{a\ell} \right).
    \end{align*}
    Hence, using $|\Gamma_a(z)| \ge 2^{-1} \alpha n k^{-1}$ (Lemma~\ref{lemma:size_predictedCommunities}), we obtain 
    \begin{align*}
      \left| \bmu_{a\ell}(z) - \bmu_{a\ell}(z^*) \right| 
      & \wle \frac{1}{\left| \Gamma_a(z) \right| } \sum_{i \in [n]} \sum_{b \in [k] \backslash\{a\} } \1 \{ z_i=a, z_i^*=b \} \left| \mu_{b\ell} - \mu_{a\ell} \right| \\ 
      & \wle \frac{2 k}{\alpha n} \Delta_{\mu,\infty} \ham(z^*, z). 
    \end{align*}

    \textit{(iii) Bounding $| \xi_{a\ell}(z) - \xi_{a\ell}(z^*) |$.} This upper-bound is more complex to derive, and is computed in Lemma~\ref{lemma:bounding_diff_xi}, which establishes that for any $z$ verifying $\ham(z^*,z) \le 2^{-1} \alpha n k^{-1}$, 
    \begin{align*}
      \max_{ \substack{ a \in [k] \\ \ell \in [d] } } \, \left| \xi_{a\ell}(z) - \xi_{a\ell}(z^*) \right| 
      \wle 2 \sqrt{ \frac{2}{\alpha} } \left( C + \sqrt{2} \right) \sqrt{ \frac{ k \ham(z^*,z) }{ n } }
    \end{align*}
   holds with probability at least $1 - 6dk e^{-\frac{C}{2} \sqrt{\frac{\alpha n}{2k}} } - 2dk e^{ - \alpha \sqrt{\frac{n}{k}} }$. 

   We conclude the proof by combining the upper-bounds obtained in steps (i), (ii) and (iii) with the decomposition~\eqref{eq:separation_error_mean_estimated}. 
 \end{proof}

 \begin{proposition}
\label{prop:quality_scaleEstimator}
 Let $\Delta_{\mu,\infty} = \max_{a,b \in [k]} \| \mu_a - \mu_b \|_\infty$ and $\Delta_{\sigma,\infty} = \max_{a,b \in [k]} \| \sigma_a - \sigma_b \|_\infty$. 
 Assume $\min_{a \in [k]} |\Gamma_a(z^*) | \ge \alpha n k^{-1}$ for some constant $\alpha > 0$. For $z \in [k]^n$ and some constant $C$, define 
 \[
  \cE_{\sigma}(z) \weq 
  \left\{ \max_{ \substack{ a \in [k] \\ \ell \in [d] } } | \hsigma_{a\ell}(z) - \sigma_{a\ell} |
   \wle 
   2 \sqrt{\frac{k}{n}} + \frac{ 2 k (\Delta_{\mu,\infty} + \Delta_{\sigma,\infty} ) }{\alpha n} \ham(z^*, z) + 2 C' \sqrt{\frac{k\ham(z^*,z)}{n}} \right\}. 
 \]
 where $C' = 2( 1 + C \sqrt{2 \alpha^{-1}})$. 
 There exists a constant $C>0$ such that the event $\bigcap\limits_{ \ham(z^*,z) \le 2^{-1} \alpha n k^{-1} } \cE_{\sigma}(z)$ holds with probability at least $ 1 - 2 k d \left( 4 e^{-\alpha \sqrt{\frac{n}{k}}} + 6 e^{- \frac{C}{2} \sqrt{ \frac{\alpha n}{2k} } } \right)$.  
\end{proposition}

\begin{proof}
 We compute, using the triangle inequality, that
 \begin{align*}
    | \hsigma_{a\ell}(z) - \sigma_{a\ell} | 
    & \weq \left| \frac{1}{ |\Gamma_{a}(z)| } \sum_{i \in \Gamma_{a}(z)} \left|X_{i\ell} - \hmu_{a\ell} \right| - \sigma_{a\ell} \right| \\
    & \wle \frac{1}{ |\Gamma_a(z)|} \left| \sum_{i \in \Gamma_a(z)} \left( |X_{i\ell} - \mu_{a\ell}| - \sigma_{a\ell} \right) \right| + \left| \hmu_{a\ell}(z) - \mu_{a\ell} \right|. 
 \end{align*}
 Proposition~\ref{prop:quality_locationEstimator} provides an upper bound on the second term of the right-hand side of the last inequality, $\left| \hmu_{a\ell}(z) - \mu_{a\ell} \right|$.

 It also provides an upper bound on the first term of the right-hand side of the last inequality. Indeed, let $\tX_{i\ell} = |X_{i\ell} - \mu_{a\ell}|$ and $\tsigma_{a\ell}(z) = \frac{1}{ |\Gamma_a(z)|} \sum_{ i \in \Gamma_a(z) } \tX_{i\ell}$. 
 The random variables $\tX_{i\ell}$ are sub-exponential, and $\tsigma_{a\ell}(z)$ is the sample mean computed over the subset $\{ \tX_{i\ell}, i \in \Gamma_a(z) \}$. Therefore we can again  apply Proposition~\ref{prop:quality_locationEstimator} to show that 
 \begin{align*}
   \max_{ \substack{ a \in [k] \\ \ell \in [d] } } 
   \left| \tsigma_{a\ell}(z) - \sigma_{a\ell} \right| 
   \wle \sqrt{\frac{k}{n}} + \frac{ 2 k \Delta_{\sigma,\infty} }{\alpha n} \ham(z^*, z) 
   + C' \sqrt{\frac{k\ham(z^*,z)}{n}}
 \end{align*}
 with probability at least $1 - k d \left( 4 e^{-\alpha \sqrt{\frac{n}{k}}} + 6 e^{- \frac{C}{2} \sqrt{ \frac{\alpha n}{2k} } }\right)$.  
\end{proof}

\subsection{Additional technical lemmas}

\begin{lemma}
\label{lemma:size_predictedCommunities}
 Let $z^* \in [k]^n$ such that $ \min_{a \in [k]} \sum_{i\in[n]} | \Gamma_a(z^*) | \ge \alpha n/ k$ for all $a \in [k]$ and for some $\alpha > 0$. Let $z \in [k]^n $ such that $\ham(z,z') \le \alpha n/(2k)$. 
 Then  
 \[
 \sum_{i=1}^n \1\{ z_i = a \cap z^*_i = a \} \wge \frac{\alpha n}{2k}.
 \]
 In particular, $\sum_{i=1}^n \1\{ z_i = a \} \ge \alpha n/(2k)$. 
\end{lemma}

\begin{proof}
 We have 
 \begin{align*}
    \sum_{i=1}^n \1\{ z_i = a \cap z^*_i = a \} 
    & \weq \sum_{i=1}^{n} \1\{ z_i^* = a\} - \sum_{i=1}^{n} \1\{ z_i^* = a \cap z_i \ne a \} \\ 
    & \wge \sum_{i=1}^{n} \1\{ z_i^* = a\} - \sum_{i=1}^{n} \1\{ z_i^* \ne z_i \} \\ 
    & \wge \frac{\alpha n}{k} - \frac{\alpha n}{2k} \\ 
    & \weq \frac{\alpha n}{2k}. 
 \end{align*}
 Finally, because $\sum_{i=1}^n \1\{ z_i = a \} \ge \sum_{i=1}^n \1\{ z_i = a \cap z^*_i = a \} $ we also have $\sum_{i=1}^n \1\{ z_i = a \} \ge \alpha n/(2k)$.
\end{proof}

 \begin{lemma}
 \label{lemma:bounding_diff_xi}
 Let $z^* \in [k]^n$ such that $\min_{a \in [k]} |\Gamma_a(z^*) | \ge \alpha n / k$ for some constant $\alpha > 0$.  For any $z \in [k]^n$ and $C \ge 1$, we define the event 
 \begin{align*}
     \cE(z) \weq 
        \left\{ 
        \max_{ \substack{ a \in [k] \\ \ell \in [d] } } | \xi_{a\ell}(z) - \xi_{a\ell}(z^*) | 
        \wle 2 \sqrt{ \frac{2}{\alpha} } \left( C + \sqrt{2} \right) \sqrt{ \frac{ k \ham(z^*,z) }{ n } } \right\}.
 \end{align*}
  There exists a constant $C \ge 1$ such that the event $\bigcap\limits_{ z \colon \ham(z^*,z) \le \frac{\alpha n}{2k} } \cE(z)$ holds with probability at least $1 - 6dk e^{-\frac{C}{2} \sqrt{\frac{\alpha n}{2k}} } - 2dk e^{ - \alpha \sqrt{\frac{n}{k} }}$.
 \end{lemma}
 \begin{proof}
 The random variables $\epsilon_{i\ell}$ are independent, zero-mean, and sub-exponential. 
 Therefore Lemma~\ref{lemma:uniformBounding_deviations_sum_subexponential_rv} and a union bound ensure the existence of a constant $C > 0$ such that the event 
 \begin{align}
    \cE_1 \weq \left\{ \max_{\ell \in [d]} \sup_{ \substack{ S \subset [n] \\ |S| \le \frac{ \alpha n }{ 2 k } } } \frac{1}{\sqrt{|S|}} \left| \sum_{i \in S} \epsilon_{i\ell} \right| \wle C \sqrt{n} \right\}
 \end{align}
 holds with a probability of at least $1 - 6 d e^{-\frac{C}{2} \sqrt{\frac{\alpha n}{2k} } }$. 
 Similarly, we have established in~\eqref{eq:upperBound_sumOverTrueClusters} that the event 
 \begin{align*}
    \cE_2(a) \weq \left\{ \max_{\ell \in [d] } \frac{1}{|\Gamma_a(z^*)|} \left| \sum_{i \in [n] } \1\{ z_i^* = a \} \epsilon_{i\ell} \right| 
     \wle \sqrt{ \frac{k}{n} } \right\},
 \end{align*}
 where $a \in [k]$, holds with probability at least $1 - 2 d e^{-\alpha \sqrt{n/k}}$. 
 Let $\cE_2 = \cap_{a \in [k]} \cE_2(a)$. We have 
 \begin{align}
 \label{eq:in_proof_eventE1capE2}
     \pr\left( \cE_1 \cap \cE_2 \right) \ge 1 - 6d e^{- C/2 \sqrt{ \alpha n / (2k) } } - 2dk e^{ - \alpha \sqrt{ n/k }}. 
 \end{align}
 
 In the rest of the proof, we work conditionally on the event $\cE_1 \cap \cE_2$, and will show that the event $\cap_{z \colon \ham(z^*,z) \le \alpha n /(2k)} \cE(z)$ holds. Let $z \in [k]^n$ verifying $\ham(z^*,z) \le \alpha n /(2k)$ and let $a \in [k]$ and $\ell \in [d]$. We have
 \begin{align*}
   & | \xi_{a\ell}(z) - \xi_{a\ell}(z^*) | \\
   \weq & \left | \frac{ \sum_{i \in [n] } \1 \{ z_i = a \} \epsilon_{i\ell} }{ \sum_{i \in [n] } \1 \{ z_i = a \} } - \frac{ \sum_{i \in [n] } \1 \{ z_i^* = a \} \epsilon_{i\ell} }{ \sum_{i \in [n] } \1 \{ z_i^* = a \} } \right| \\
   \wle & 
   \underbrace{ \left| \frac{ \sum_{i \in [n] } \1 \{ z_i = a \} \epsilon_{i\ell} }{ \sum_{i \in [n] } \1 \{ z_i = a \} } - \frac{ \sum_{i \in [n] } \1 \{ z_i^* = a \} \epsilon_{i\ell} }{ \sum_{i \in [n] } \1 \{ z_i = a \} }  \right| }_{ E_1 }
   + 
   \underbrace{ \left| \frac{ \sum_{i \in [n] } \1 \{ z_i^* = a \} \epsilon_{i\ell} }{ \sum_{i \in [n] } \1 \{ z_i = a \} } - \frac{ \sum_{i \in [n] } \1 \{ z_i^* = a \} \epsilon_{i\ell} }{ \sum_{i \in [n] } \1 \{ z_i^* = a \} }  \right| }_{E_2}.
 \end{align*}
 Let us first upper bound $E_1$. We have 
 \begin{align*}
  E_1 & \weq \frac{1}{|\Gamma_a(z)|} \left| \sum_{i \in [n] } \left( \1 \{ z_i=a\} - \1 \{z_i^* = a \} \right) \epsilon_{i\ell} \right| \\
  & \wle 
  \underbrace{ \frac{1}{|\Gamma_a(z)|} \left| \sum_{i \in [n] } \1 \{ z_i=a, z_i^* \ne a \} \epsilon_{i\ell} \right| }_{E_{11}} + \underbrace{\frac{1}{|\Gamma_a(z)|} \left| \sum_{i \in [n] } \1 \{ z_i \ne a, z_i^* = a \} \epsilon_{i\ell} \right| }_{ E_{12} }. 
 \end{align*}
 Let us denote by $\Gamma_a^c(z^*) = [n] \backslash \Gamma_a(z^*)$ the complement of $\Gamma_a(z^*)$. Noticing that $|\Gamma_{a}(z)| \ge \alpha n/(2k)$ (because of Lemma~\ref{lemma:size_predictedCommunities}) and that
 $$ |\Gamma_{a}(z) \cap \Gamma_a^c(z^*)| \weq \sum_{i \in [n] } \1 \{ z_i = a, z_i^* \ne a \} \wle \ham(z^*, z) \wle  \frac{\alpha n}{2k},$$
we obtain  
 \begin{align*}
   E_{11} 
   & \weq \frac{1}{|\Gamma_{a}(z)|} \sqrt{ |\Gamma_{a}(z) \cap \Gamma_a^c(z^*)| } \cdot \frac{ \left| \sum_{i \in [n] } \1 \{ z_i = a, z_i^* \ne a \} \epsilon_{i\ell} \right| }{ \sqrt{ |\Gamma_{a}(z) \cap \Gamma_a^c(z^*)| } } \\
   & \wle \frac{2k}{\alpha n} \sqrt{ \ham(z^*, z) } \, \sup_{ \substack{ S \subset [n] \\ |S| \le 2^{-1} \alpha n k^{-1} } } \frac{ 1 }{ \sqrt{|S|} } \left| \sum_{i \in S } \epsilon_{i\ell} \right|.
 \end{align*}
 Conditioning $E_{11}$ on the event~$\cE_1$, we have therefore  
 \begin{align*}
    E_{11} \wle \frac{2k}{\alpha n} \sqrt{ \ham(z^*, z) } \cdot C \sqrt{ \frac{\alpha n}{2k} } = C \sqrt{\frac{2k}{\alpha n} \, \ham(z^*, z) }  . 
 \end{align*}
 Proceeding similarly, we establish the same upper bound holds for $E_{12}$. Therefore, conditionally on $\cE_1$, we have 
 \begin{align}
 \label{eq:in_proof_E1}
    E_1 \wle 2C \sqrt{\frac{2k}{\alpha n} \, \ham(z^*, z) }. 
 \end{align}
 We can now upper-bound $E_2$, whose expression can be recast as 
 \begin{align*}
   E_2 
   & \weq \underbrace{ \left| \frac{\sum_{ i \in [n]} \1 \{ z_i^* = a \} - \1 \{ z_i = a \} }{ \sum_{i \in [n] } \1 \{ z_i = a \} } \right| }_{E_{21}} \cdot 
   \underbrace{ \left| \frac{ \sum_{i \in [n] } \1 \{ z_i^* = a \} \epsilon_{i\ell} }{ \sum_{i \in [n] } \1 \{ z_i^* = a \} } \right| }_{ E_{22} }. 
 \end{align*}
 We have 
 \begin{align*}
    \left| \sum_{ i \in [n]} \1 \{ z_i^* = a \} - \1 \{ z_i = a \} \right|
    & \weq \left| \sum_{i\in[n]} \1\{ z_i^* = a, z_i \ne a \} - \sum_{i\in[n]} \1\{ z_i^* \ne a, z_i = a \} \right| \\
    & \wle \max \left \{ \left| \sum_{i\in[n]} \1\{ z_i^* = a, z_i \ne a \} \right|, \left| \sum_{i\in[n]} \1\{ z_i^* \ne a, z_i = a \} \right| \right\} \\
    & \wle \ham(z^*,z).
 \end{align*}
 Moreover, Lemma~\ref{lemma:size_predictedCommunities} yields that $ \sum_{i \in [n] } \1 \{ z_i = a \} \ge \sum_{i \in [n] } \1 \{ z_i=a, z_i^* = a \} \ge \alpha n/(2k) $. Therefore, 
 \begin{align*}
    E_{21} \wle \frac{2 k}{\alpha n} \ham(z^*, z). 
 \end{align*}
 Finally, conditionally on the event $\cE_2$, we have $E_{22} \le \sqrt{k/n}$. 
 Using $ \sqrt{\ham(z^*,z)} \le \sqrt{ \alpha n / (2k)}$, we obtain  
 \begin{align}
  \label{eq:in_proof_E2}
   E_2 \wle \frac{2 k}{\alpha n} \sqrt{ \frac{k}{n} } \, \ham(z^*, z) 
   \wle \frac{4}{\sqrt{2\alpha}} \sqrt{\frac{k \ham(z^*,z)}{n} }. 
 \end{align}
 We conclude the proof by combining~\eqref{eq:in_proof_E1} and~\eqref{eq:in_proof_E2} and using~\eqref{eq:in_proof_eventE1capE2}. 
\end{proof}

\section{Proofs for Section~\ref{subsection:LaplaceMixtures} (Laplace mixture models)}
\label{appendix:proofLaplaceMixture}

\subsection{Proof of Theorem~\ref{thm:LaplaceMixtureRecovery}}
 Following the general result on recovering parametric mixture models (Lemma~\ref{lemma:generalParametricMixtureRecovery}), to prove  Theorem~\ref{thm:LaplaceMixtureRecovery}, we need to show that Condition~\ref{condition:excess_error} holds. We will show that it does so for arbitrary $c,c'$ (which can be taken as small as we would like). 

 Because $\log f_{(\mu,\sigma)}(x) = \sum_{\ell=1}^d \left( - \log \sigma_{\ell} - \left| \frac{x-\mu_{\ell} }{ \sigma_{\ell} } \right| \right)$, we have for any $a \in [k]$ and $X_i \in \R^d$, 
\begin{align*}
 \left| \hell_{a}^{(t)}(X_i) - \ell_a(X_i) \right| 
 \weq & \left| \sum_{\ell=1}^d \left( - \log \frac{ \hsigma_{a\ell}^{(t)} }{ \sigma_{a\ell} } - \left| \frac{X_{i\ell} - \hmu_{a\ell}^{(t)} }{ \hsigma_{a\ell}^{(t)} } \right| + \left| \frac{X_{i\ell} - \mu_{a\ell} }{ \sigma_{a\ell} }\right| \right) \right| \\
 \wle & \sum_{\ell=1}^d \left( \left| \log \frac{ \hsigma_{a\ell}^{(t)} }{ \sigma_{a\ell} } \right| + \left| \left| \frac{X_{i\ell} - \mu_{a\ell} }{ \sigma_{a\ell} }\right|  - \left| \frac{X_{i\ell} - \hmu_{a\ell}^{(t)} }{ \hsigma_{a\ell}^{(t)} } \right| \right| \right). 
\end{align*}
Moreover, using $\big| |x| - |y| \big| \le |x-y|$, we have  
\begin{align*}
 \left| \bigg| \frac{X_{i\ell} - \mu_{a\ell} }{ \sigma_{a\ell} } \bigg|  - \bigg| \frac{X_{i\ell} - \hmu_{a\ell}^{(t)} }{ \hsigma_{a\ell}^{(t)} } \bigg| \right| 
 \wle & \frac{ \left| \hsigma_{a\ell}^{(t)} \left( X_{i\ell} - \mu_{a\ell} \right) - \sigma_{a\ell} \left( X_{i\ell} - \hmu_{a\ell}^{(t)} \right) \right|}{  \sigma_{a\ell} \hsigma_{a\ell}^{(t)} } \\
 \wle & \frac{ \left| \hsigma_{a\ell}^{(t)} - \sigma_{a\ell} \right| }{ \sigma_{a\ell} \hsigma_{a\ell}^{(t)} } \cdot \left| X_{i\ell} - \mu_{a\ell} \right| + \frac{ \left| \mu_{a\ell} - \hmu_{a\ell}^{(t)} \right| }{ \hsigma_{a\ell}^{(t)} } \\
 \wle & \frac{ \left| \hsigma_{a\ell}^{(t)} - \sigma_{a\ell} \right| }{ \sigma_{a\ell} \hsigma_{a\ell}^{(t)} } \cdot \Big\{ \left( \left| X_{i\ell} - \mu_{z_i^*\ell} \right| \right) 
 + \left|\mu_{z_i^*\ell} - \mu_{a\ell}\right| \Big\} 
 + \frac{ \left| \mu_{a\ell} - \hmu_{a\ell}^{(t)} \right| }{ \hsigma_{a\ell}^{(t)} } \\
 \wle & \frac{ \left| \hsigma_{a\ell}^{(t)} - \sigma_{a\ell} \right| }{ \sigma_{a\ell} \hsigma_{a\ell}^{(t)} } \cdot \left( \left| X_{i\ell} - \mu_{z_i^*\ell} \right| - \sigma_{z_i^*\ell} \right) \\
 & \qquad +  \frac{ \left| \hsigma_{a\ell}^{(t)} - \sigma_{a\ell} \right| }{ \hsigma_{a\ell}^{(t)} } \left( \frac{ \Delta_{\mu,\infty} }{ \sigma_{a\ell} } + \max_{ b, c \in [k] } \frac{\sigma_{b\ell}}{\sigma_{c\ell} } \right)
 +  \frac{ \left| \mu_{a\ell} - \hmu_{a\ell}^{(t)} \right| }{ \hsigma_{a\ell}^{(t)} }. 
\end{align*}
 Combining these upper bounds with the definition of the excess error in~\eqref{eq:def_excess_error}, we obtain 
\begin{align*}
 \xi_{\excess}(\delta) \wle 2 \delta^{-1} \left( F \cdot \ham\left( z^*, \hz^{(t)} \right) + G \right), 
\end{align*}
where $F$ and $G$ are defined by
\begin{align*}
 \tF & \weq \max_{ a \in [k] } \, \sum_{\ell=1}^d \left| \log \frac{ \hsigma_{a\ell}^{(t)} }{ \sigma_{a\ell} } \right| + \frac{ \left| \mu_{a\ell} - \hmu_{a\ell}^{(t)} \right| + \left| \hsigma_{a\ell}^{(t)} - \sigma_{a\ell} \right| \left( \sigma_{a\ell}^{-1} \Delta_{\mu,\infty} + m_{\sigma} \right) }{ \hsigma_{a\ell}^{(t)} } \\
 \tG & \weq \max_{ a \in [k] } \, \sum_{\ell = 1}^d \, \frac{ \left| \hsigma_{a\ell}^{(t)} - \sigma_{a\ell} \right| }{ \sigma_{a\ell} \hsigma_{a\ell}^{(t)} } \cdot \sum_{i \in [n]}  \left( \left| X_{i\ell} - \mu_{ z_i^* \ell} \right| - \sigma_{z_i^*\ell} \right), 
\end{align*}
where $\Delta_{\sigma, \infty} = \max_{ \substack{ 1 \le a \ne b \le k \\ \ell \in [d] } } | \sigma_{a\ell} - \sigma_{b\ell} | $ and $m_{\sigma} = \max_{1 \le a \ne b \le k } \frac{\sigma_{a\ell}}{ \sigma_{b\ell}}$. 

The quantities $\tF$ and $\tG$ are in fact functions of the true parameters $\mu, \sigma$ as well as the estimated ones $\hmu^{(t)}, \hsigma^{(t)}$. Because these estimations are made based on the data points $X$ and the predicted clusters are step $t-1$, we have $\tF = F(X,\mu,\sigma, \hz^{(t-1)})$ and $\tG = G( X, \mu, \sigma, \hz^{(t-1)})$, 
where for any $z \in [k]^n$, we define
\begin{align*}
 F(X,\mu,\sigma, z) & \weq \max_{ a \in [k] } \, \sum_{\ell=1}^d \left| \log \frac{ \hsigma_{a\ell}(z) }{ \sigma_{a\ell} } \right| + \frac{ \left| \mu_{a\ell} - \hmu_{a\ell}(z) \right| + \left| \hsigma_{a\ell}(z) - \sigma_{a\ell} \right| \left( \sigma_{a\ell}^{-1} \Delta_{\mu,\infty} + m_{\sigma} \right) }{ \hsigma_{a\ell}(z) } \\
 G(X,\mu,\sigma, z) & \weq \max_{ a \in [k] } \, \sum_{\ell = 1}^d \, \frac{ \left| \hsigma_{a\ell}(z) - \sigma_{a\ell} \right| }{ \sigma_{a\ell} \hsigma_{a\ell}(z) } \cdot \sum_{i \in [n]} \left( \left| X_{i\ell} - \mu_{ z_i^* \ell} \right| - \sigma_{z_i^*\ell} \right).
\end{align*}
The quantities $F$ and $G$ are analyzed in Lemmas~\ref{lemma:bounding_F_LaplaceMixture} and~\ref{lemma:bounding_G_LaplaceMixture}. 
In particular, under the assumptions of Theorem~\ref{thm:LaplaceMixtureRecovery}, we have (with probability at least $1-o(1))$ 
\begin{align*}
 \tF \weq o \left( d \Delta_{\mu,\infty} \right) \quad \text{ and } \quad
 \tG \weq O \left( \omega_n d \sqrt{k} \left( 1 + \frac{ \Delta_{\mu,\infty} }{ \sqrt{ nk^{-1} } } \right) \right) \times \max \left\{  \ham\left( z^*, z^{(t-1)} \right) , 1 \right\}, 
\end{align*}
where we are free to choose the sequence $\omega_n$ as long as $\omega_n = \omega( 1 )$. 

Suppose $z^{(t-1)} \ne z^*$. We choose $\delta$ such that $\delta = o( \Chernoff(\cF) )$ and $ \omega_n d \sqrt{k} ( 1 + \frac{ \Delta_{\mu,\infty} }{ \sqrt{ nk^{-1} } }) = o(\delta)$. Such a choice is possible. Indeed, because by assumption $\Chernoff(\cF) = \omega( d \sqrt{k} ( 1 + \frac{ \Delta_{\mu,\infty} }{ \sqrt{ nk^{-1} } }) )$, we can write $\Chernoff(\cF) = d \sqrt{k} ( 1 + \frac{ \Delta_{\mu,\infty} }{ \sqrt{ nk^{-1} } }) \tau_n$ with $\tau_n = \omega(1)$. Then, we can choose $\omega_n = \tau_n^{1/4}$ and $\delta = d \sqrt{k} ( 1 + \frac{ \Delta_{\mu,\infty} }{ \sqrt{ nk^{-1} } }) \sqrt{ \tau_n}$. 
With this particular choice of $\delta$, we have $\xi_{\excess}^{(t)}(\delta) = o\left( \ham\left( z^*, \hz^{(t)} \right) \right) + o \left( \ham\left( z^*, \hz^{(t-1)} \right) \right) $. Hence, 
\begin{align*}
 \xi_{\excess}^{(t)}(\delta) \wle c \ham\left( z^*, \hz^{(t)} \right) + c' \ham\left( z^*, \hz^{(t-1)} \right)  
\end{align*}
for arbitrary constants $c, c' > 0$. This establishes Condition~\ref{condition:excess_error}. 

Finally, suppose that $z^{(t-1)} = z^*$. By choosing $\delta$ as in the previous paragraph, we have $\xi_{\excess}^{(t)}(\delta) = o(\ham(z^*,z^{(t)}) + o(1)$ and therefore 
\begin{align}
\label{eq:in_proof_last_case}
 (1+o(1)) \ham(z^*,z^{(t)}) \wle n e^{-(1+o(1))\Chernoff(\cF)} + o(1).
\end{align}
If $n e^{-(1+o(1))\Chernoff(\cF)} ) = o(1)$ then $\ham(z^*,z^{(t)}) = o(1)$ because $\ham(z^*,z^{(t)})$ is integer. Otherwise, if $n e^{-(1+o(1))\Chernoff(\cF)}$ is bounded away from $0$, then the $o(1)$ in the right hand side of~\eqref{eq:in_proof_last_case} can be absorbed by the term $n e^{-(1+o(1))\Chernoff(\cF)} $. This implies 
$\ham(z^*,z^{(t)}) \wle n e^{-(1+o(1))\Chernoff(\cF)} $, and this ends the proof. 

\subsection{Bounding \texorpdfstring{$F$}{F} and \texorpdfstring{$G$}{G}}

\begin{lemma}
 \label{lemma:bounding_F_LaplaceMixture} Suppose that $\min_{a \in [k]} |\Gamma_a(z^*)| \ge \alpha n/ k$ for some $\alpha >0$. 
 Assume also that $\sigma_{a\ell} = \Theta(1)$ for all $a\in [k], \ell \in [d]$. Let $\epsilon > 0$ and $\tau = o( n k^{-1} \Delta_{\mu,\infty}^{-1} )$. For $n$ large enough, we have  
 \begin{align*}
  \max_{ \substack{ z \in [k]^n \\ \ham(z^*, z) \le \tau } } F(X,\mu,\sigma,z) \weq o( d \Delta_{\mu,\infty} )
 \end{align*}
 with probability at least $ 1 - 3 k d \left( 4 e^{-\alpha \sqrt{\frac{n}{k}}} + 6 e^{- \frac{C}{2} \sqrt{ \frac{\alpha n}{2k} } } \right) $. 
\end{lemma}

\begin{proof}
 Let $\cZ_{\tau} = \{ z \in [k]^n \colon \ham(z^*, z ) \le \tau \}$. In the rest of the proof, we work conditionally on the event $\cE = \cE_{\mu} \cap \cE_{\sigma}$, where 
\begin{align*}
 \cE_{\mu} \weq & \left\{ 
 \max_{ \substack{ z \in \cZ_{\tau} \\ a \in k \\ \ell \in [d] } } | \hmu_{a\ell}(z) - \mu_{a\ell} |  \wle \sqrt{\frac{k}{n}} +  \frac{ 2k \Delta_{\mu,\infty} }{ \alpha n }  \tau + C'\sqrt{ \frac{ k\tau }{ n } } 
   \right\}, \\
 \cE_{\sigma} \weq & \left\{ 
 \max_{ \substack{ z \in \cZ_{\tau} \\ a \in k \\ \ell \in [d] } } | \hsigma_{a\ell}(z) - \sigma_{a\ell} |
   \wle 
   2 \sqrt{\frac{k}{n}} + \frac{ 2 k (\Delta_{\mu,\infty} + \Delta_{\sigma,\infty} ) }{\alpha n} \tau
   + 2 C' \sqrt{\frac{ k \tau }{ n } }
   \right\}. 
\end{align*}
where $C' = 2(1+\sqrt{2\alpha^{-1}})$. 
 By combining Proposition~\ref{prop:quality_locationEstimator} and~\ref{prop:quality_scaleEstimator}, the event $\cE = \cE_{\mu} \cap \cE_{\sigma}$ holds with a probability of at least 
 $$ \pr(\cE) \wge 1 - 3 k d \left( 4 e^{-\alpha \sqrt{\frac{n}{k}}} + 6 e^{- \frac{C}{2} \sqrt{ \frac{\alpha n}{2k} } } \right). $$ We will show that, conditioned on this event $\cE$, we have $\max\limits_{ z \in \cZ_{\tau} } F(X,\mu,\sigma,z) = o(d \Delta_{\mu,\infty})$. 
 We first notice that 
 \begin{align}
 \label{eq:in_proof_decomposisiontF}
   \max_{ z \in \cZ_{\tau} } F(X,\mu,\sigma,z) 
   \wle & d \Bigg( \max_{ \substack{ z \in \cZ_{\tau} \\ a \in [k] \\ \ell \in [d] } } \left| \log \frac{ \hsigma_{a\ell}(z) }{ \sigma_{a\ell} } \right|  
   +  \frac{ \left| \mu_{a\ell} - \hmu_{a\ell}(z) \right| }{ \hsigma_{a\ell}(z) } 
   + \frac{ \left| \hsigma_{a\ell}^{(t)} - \sigma_{a\ell} \right| }{ \hsigma_{a\ell}^{(t)} } \cdot \left( \sigma_{a\ell}^{-1} \Delta_{\mu,\infty} + m_{\sigma} \right) \Bigg)
 \end{align}
Moreover, using the event $\cE_{\sigma}$ and $\tau = o(n k^{-1})$, we have $| \hsigma_{a\ell} - \sigma_{a\ell} | = o(1)$ for any $z \in \cZ_{\tau}$. 
Let~$n$ be large enough so that $| \hsigma_{a\ell} - \sigma_{a\ell} | \sigma_{a\ell}^{-1} \le 2^{-1}$ (notice this large enough $n$ does not depend on $z$). Using $| \log(1+t) | \le \frac{|t|}{1-|t|}$ for any $|t| < 1$, we have 
\begin{align*}
 \left| \log \frac{ \hsigma_{a\ell}(z) }{ \sigma_{a\ell} } \right| 
 \weq \left| \log \left( 1 + \frac{ \hsigma_{a\ell}(z) - \sigma_{a\ell} }{\sigma_{a\ell}} \right) \right| 
 \wle 2 \left| \frac{ \hsigma_{a\ell}(z) - \sigma_{a\ell} }{ \sigma_{a\ell} } \right|.
\end{align*}
 Because $\sigma_{a\ell} = \Theta(1)$, this ensures that 
 \begin{align}
   \label{eq:in_proof_firstTermF}
   \max_{ \substack{ z \in \cZ_{\tau} \\ a \in [k] \\ \ell \in [d] } } \left| \log \frac{ \hsigma_{a\ell}(z) }{ \sigma_{a\ell} } \right| \weq o(1). 
 \end{align}
 Similarly, 
 \begin{align}
 \label{eq:in_proof_secondTermF}
   \max_{ \substack{ z \in \cZ_{\tau} \\ a \in [k] \\ \ell \in [d] } } \frac{ \left| \mu_{a\ell} - \hmu_{a\ell}(z) \right| }{ \hsigma_{a\ell}(z) } \weq o(1),
 \end{align}
 and because $ \sigma_{a\ell} = \Theta(1)$ and $m_{\sigma} = \Theta(1)$ we also have  
 \begin{align}
 \label{eq:in_proof_thirdTermF}
   \max_{ \substack{ z \in \cZ_{\tau} \\ a \in [k] \\ \ell \in [d] } } \frac{ \left| \hsigma_{a\ell}^{(t)} - \sigma_{a\ell} \right| }{ \hsigma_{a\ell}^{(t)} } \cdot \left( \sigma_{a\ell}^{-1} \Delta_{\mu,\infty} + m_{\sigma} \right) \weq o\left( \Delta_{\mu,\infty} \right). 
 \end{align}
 We conclude by combining~\eqref{eq:in_proof_decomposisiontF} with~\eqref{eq:in_proof_firstTermF},~\eqref{eq:in_proof_secondTermF}, and~\eqref{eq:in_proof_thirdTermF}. 
\end{proof}

\begin{lemma}
\label{lemma:bounding_G_LaplaceMixture}
Let $\epsilon > 0$, $\tau = o( n/ k )$. Let $\sigma_{\min} = \min_{ \substack{ a \in [k], \ell \in [d] } } \sigma_{a\ell}$ and $\omega_n = \omega(1)$. 
Suppose that $\min_{a \in [k]} |\Gamma_a(z^*)| \ge \alpha n/ k$ for some $\alpha >0$. For $n$ large enough, it holds 
 \begin{align*}
   \max_{ \substack{ z \in [k]^n \\ 1 \le \ham(z^*, z) \le \tau } } \frac{ G(X,\mu,\sigma,z) }{ \ham(z^*, z) } 
   \weq O \left( d \sqrt{k} \omega_n \left( 1 + \frac{ \Delta_{\mu,\infty} }{ \sqrt{ n/k } } \right) \right) 
 \end{align*}
 and
 \begin{align*}
    G(X,\mu,\sigma,z^*) \weq O \left( d \sqrt{k} \omega_n \right)
 \end{align*}
 with probability at least $ 1 - 2e^{-\omega_n^2} - 3 k d \left( 4 e^{-\alpha \sqrt{ \frac{n}{k}}} + 6 e^{- \frac{C}{2} \sqrt{ \frac{\alpha n}{2k} } } \right)$. 
\end{lemma}

\begin{proof}
Let $\cZ_{\tau} = \{ z \in [k]^n \colon 1 \le \ham(z^*, z ) \le \tau \}$. 
 Because the random variables $X_{i\ell}$ are Laplace distributed with location $\mu_{z_i^*\ell}$ and scale $\sigma_{z_i^*\ell}$, the random variables $2 \sigma_{z_i^*\ell}^{-1} \left| X_{i\ell} - \mu_{z_i^*\ell} \right|$ are $\chi^2(2)$ distributed. Hence, the random variables $ Y_i = \left| X_{i\ell} - \mu_{z_i^*\ell} \right| - \sigma_{z_i^*\ell}$ are sub-exponential with zero mean. Bernstein's inequality for sub-exponential random variables (see~\eqref{eq:Bernstein_subexponential}) ensures that the event 
\begin{align*}
 \cE_{1} \weq \left\{ \sum_{i \in [n]} \left( \left| X_{i\ell} - \mu_{z_i^*\ell} \right| - \sigma_{z_i^*\ell} \right) 
 \wle \omega_n \sqrt{n} \right\} 
\end{align*}
 holds with a probability of at least $1 - 2e^{-\omega_n^2}$. Moreover, let $\cE_{\mu}(z)$ and $\cE_{\sigma}(z)$ be the events 
\begin{align*}
 \cE_{\mu}(z) \weq & \left\{ 
 \max_{ \substack{ a \in k \\ \ell \in [d] } } | \hmu_{a\ell}(z) - \mu_{a\ell} |  \wle \sqrt{\frac{k}{n}} + \frac{ 2k \Delta_{\mu,\infty} }{ \alpha n } \ham(z^*, z) + C' \sqrt{ \frac{ k \ham(z^*, z) }{ n } } 
   \right\}, \\
 \cE_{\sigma}(z) \weq & \left\{ 
 \max_{ \substack{ a \in k \\ \ell \in [d] } } | \hsigma_{a\ell}(z) - \sigma_{a\ell} |
   \wle 
   2 \sqrt{\frac{k}{n}} + \frac{ 2 k (\Delta_{\mu,\infty} + \Delta_{\sigma,\infty} ) }{ \alpha n } \ham(z^*, z)
   + 2 C' \sqrt{\frac{ k \ham(z^*, z) }{ n } }
   \right\}. 
\end{align*}
By Propositions~\ref{prop:quality_locationEstimator} and~\ref{prop:quality_scaleEstimator}, the event $\cE = \cE_1 \bigcap\limits_{ z \in \cZ_{\tau} } \left( \cE_{\mu}(z) \cap \cE_{\sigma} (z) \right)$ holds with probability at least 
$
 1 - 2e^{-\omega_n^2} - 3 k d \left( 4 e^{-\alpha \sqrt{\frac{n}{k}}} + 6 e^{- \frac{C}{2} \sqrt{ \frac{\alpha n}{2k} } } \right).
$
On this event $\cE$, we have for all $z \in \cZ_{\tau}$, 
\begin{align*}
 \frac{ G(X,\mu,\sigma,z) }{ \ham(z^*, z) } 
 & \wle \frac{ d \omega_n }{ \min\limits_{ \substack{ a \in [k] \\ \ell \in [L] } } \sigma_{a\ell}^2 } \left( \frac{2 \sqrt{k} }{ \ham(z,z^*) } + \frac{2k(\Delta_{\mu,\infty} + \Delta_{\sigma,\infty})}{ \alpha \sqrt{n} } + 2 C' \sqrt{\frac{k}{\ham(z^*,z)}} \right),
\end{align*}
and the result holds by noticing that $\ham(z^*,z) \ge 1$. 

To obtain the upper bound on $G(X,\mu,\sigma,z^*)$, we use the events $\cE_1$ and $\cE_{\sigma}(z^*)$. 
\end{proof}

\subsection{Example}
\label{appendix:detailed_exampleLaplace}

In this section, we consider a Laplace mixture model where for all $\ell \in [d]$ we have $\sigma_{
 1\ell} = \cdots = \sigma_{k\ell}$, and we simply denote this quantity by $\sigma_{\ell}$. Assume also that $\sigma_{\ell}$ is independent of $n$, and that $\mu_{a\ell} = m_{a \ell} \rho_n$ where $m_{a\ell}$ are non-zero constants independent of $n$, verifying $m_{a\ell} \ne m_{b\ell}$ whenever $a\ne b$, and $\rho_n = \omega(1)$. The following Lemma provides the expression of the Chernoff information of this model. 

 \begin{lemma}
  Let $\cF= \{f_1, \cdots, f_k\}$ be the mixture of Laplace distributions described above. We have 
   $\Chernoff(\cF) = (1+o(1)) \sum_{\ell=1}^d \frac{ | \mu_{a\ell} - \mu_{b\ell} | }{ \sigma_{\ell} }$.
 \end{lemma}
 \begin{proof}
   We denote by $\dren_t(f\|g)$ the \Renyi divergence of order $t$ between $f$ and $g$, and we recall that $\Chernoff(f,g) = \sup_{t \in (0,1)} (1-t) \dren_t(f \| g)$. 
   From direct computations of Renyi divergence between Laplace distributions (see for example~\cite{gil2013renyi}), we observe that $t \mapsto (1-t) \dren_t(f_a \| f_b)$ is maximal at $t=1/2$, and we further have
   \begin{align*}
    \Chernoff(f_a, f_b) 
    & \weq \frac12 \sum_{\ell=1}^d \left( \frac{ | \mu_{a\ell} - \mu_{b\ell} | }{ \sigma_{\ell} } - 2 \log \left( 1 + \frac{ | \mu_{a\ell} - \mu_{b\ell} | }{ 2 \sigma_{\ell} } \right) \right)
   \end{align*}
   and we conclude because $ | \mu_{a\ell} - \mu_{b\ell} | = \omega(1)$ and $\sigma_{\ell} = \Theta(1)$. 
 \end{proof}

 Hence, we have $\Delta_{\mu,\infty} = \Theta(\rho_n)$ and $\Chernoff(\cF) = \Theta(d \rho_n)$. 
 Therefore, the conditions $\Delta_{\mu,\infty} = O( d^{-1} \Chernoff(\cF) )$ and $\Chernoff(\cF) = \omega(d \sqrt{k})$ of Theorem~\ref{thm:LaplaceMixtureRecovery} become  $\rho_n = \omega(\sqrt{k})$. 
 Moreover, for the initialisation, the condition of Lemma~\ref{lemma:consistency_initialisation} are verified if $\rho_n = \omega(k^2)$. 
 

\section{Proofs for Section~\ref{subsection:exponentialFamilyMixtures} (Exponential family mixture models)}
\label{appendix:proofexpofamilyMixture}

\subsection{Proof of Theorem~\ref{thm:expofamilyMixtureRecovery}}

 To prove Theorem~\ref{thm:expofamilyMixtureRecovery}, we adopt the same approach as in the proof of Theorem~\ref{thm:LaplaceMixtureRecovery} done in Section~\ref{appendix:proofLaplaceMixture}. 
 To simplify the notations, we suppose that the data points $X_i$ are sampled from a natural exponential family, that is, the sufficient statistics $u$ in the definition of the exponential family~\eqref{eq:def_exponential_family} is the identity. The proof for a general exponential family is obtained by substituting~$u(X_i)$ for $X_i$ throughout the proof. 
 
 We first notice that for any $a \in [k]$, we have
\begin{align*}
 & \left| \hell_{a}^{(t)}(X_i) - \ell_a(X_i) \right| \\
 \weq & \left| \sum_{\ell=1}^d \dbreg_{ \psi^*_{\ell} } \left( X_{i\ell}, \hmu_{a\ell}^{(t)} \right) - \dbreg_{ \psi^*_{\ell} } \left( X_{i\ell}, \mu_{a\ell} \right) \right| \\
 \weq & \left| \sum_{\ell=1}^d \dbreg_{\psi^*_{\ell} } \left( \mu_{a\ell}, \hmu_{a\ell}^{(t)} \right) + \langle X_{i\ell} - \mu_{a\ell}, \nabla\psi_{\ell}^*( \mu_{a\ell} ) - \nabla\psi_{\ell}^*( \hmu_{a\ell}^{(t)} ) \rangle \right| \\
 \wle & \sum_{\ell=1}^d \left\{ \left| \dbreg_{ \psi^*_{\ell} } \left( \mu_{a\ell}, \hmu_{a\ell}^{(t)} \right) \right| + \left( | X_{i\ell} - \mu_{z_i^* \ell } | + | \mu_{z^*_i \ell } - \mu_{a\ell} | \right) \cdot \left| \nabla \psi_{\ell}^*( \mu_{a\ell} ) - \nabla \psi_{\ell}^*\left( \hmu_{a\ell}^{(t)} \right) \right| \right\},
\end{align*}
where the second equality uses Lemma~\ref{lemma:difference_breg_divergences} and the last inequality uses the triangle and Cauchy-Schwarz's inequalities. 
Hence, combining this bound with the definition of the excess error~\eqref{eq:def_excess_error}, we obtain 
\begin{align*}
 \xi_{\excess}(\delta) 
 & \wle 2 \delta^{-1} \left( F \cdot \ham\left( z^*, z^{(t)} \right) + G \right)
\end{align*}
where $F = F\left( X, \mu, \sigma, z^{(t-1)} \right)$ and $G = G\left( X, \mu, \sigma, z^{(t-1)} \right)$ are defined by 
\begin{align*}
 F(X,\mu,\sigma, z) & \weq d \max_{ \substack{a \in [k] \\ \ell \in [d] } } \left\{ \left| \dbreg_{\psi_{\ell}^* } \left( \mu_{a\ell}, \hmu_{a\ell}(z) \right) \right| + \left( | \mu_{z^*_i \ell} - \mu_{a\ell} | + \sigma_{z_i^* \ell} \right) \cdot \left| \nabla \psi_{\ell}^*(\mu_{a\ell} ) - \nabla \psi_{\ell}^*\left( \hmu_{a\ell}(z) \right) \right| \right\}, \\
 G(X,\mu,\sigma, z) & \weq \sum_{\ell=1}^d \sum_{ i \in [n] } \sum_{b \in [k]} \left( | X_{i\ell} - \mu_{z_i^* \ell} | - \sigma_{z_i^* \ell} \right) \cdot \max_{a \in [k]} \left| \nabla \psi_{\ell}^*(\mu_{a\ell}) - \nabla \psi_{\ell}^*\left( \hmu_{a\ell}(z) \right) \right|,  
\end{align*}
where $\sigma_{z_i^* \ell} = \E \left[ \left| X_{i\ell} - \mu_{z_i^*\ell} \right| \right]$. The bounding of the quantities $F$ and $G$ is done in a very similar manner as what is done in Lemmas~\ref{lemma:bounding_F_LaplaceMixture} and~\ref{lemma:bounding_G_LaplaceMixture}. Indeed, let $\cZ_{\tau} = \{ z \in [k]^n \colon \ham(z^*, z ) \le \tau \}$. Because by assumption the $X_{i\ell}$ are sub-exponential, we can apply Proposition~\ref{prop:quality_locationEstimator} to show that the event $\cE = \bigcap_{\ham(z,z^*) \le 2^{-1} \alpha n k^{-1}} \cE_{\mu}(z)$, where  
\begin{align*}
 \cE_{\mu}(z) \weq & \left\{ 
 \max_{ \substack{ a \in k \\ \ell \in [d] } } | \hmu_{a\ell}(z) - \mu_{a\ell} |  \wle \sqrt{\frac{k}{n}} + \frac{ 2k \Delta_{\mu,\infty} }{ \alpha n } \ham(z^*, z) + C' \sqrt{ \frac{ k \ham(z^*, z) }{ n } } 
   \right\} 
\end{align*}
 holds with a probability of at least 
 $ 1 - k d \left( 4e^{-\alpha \sqrt{nk^{-1}}} + 6 e^{-C \sqrt{\alpha 2^{-1} n k^{-1}}} \right). $ Under this event, we notice that $|\mu_{a\ell} - \hmu_{a\ell}(z)| = o(1)$ and we bound $F$ as follows. We first use Lemma~\ref{lemma:bound_bregman_sequenceConverging} to show that
  \begin{align*}
   \left| \dbreg_{ \psi_{\ell}^* }( \mu_{a\ell}, \hmu_{a\ell}(z) ) \right| \wle \left| \nabla^2 \psi_{\ell}^*(\mu_{a\ell}) \right| \cdot | \hmu_{a\ell}(z) - \mu_{a\ell} |^2.
 \end{align*}
 Similarly, 
 \begin{align*}
   \left| \nabla \psi_{\ell}^*(\mu_{a\ell} ) - \nabla \psi_{\ell}^*\left( \hmu_{a\ell}(z) \right) \right|
   & \wle \sup_{y \in [\mu_{a\ell}, \hmu_{a\ell}(z)] } \left| \nabla^2 \psi_{\ell}^*(y) \right| \cdot | \hmu_{a\ell}(z) - \mu_{a\ell} | \\
   & \wle 2 \left| \nabla^2 \psi_{\ell}^*(\mu_{a\ell}) \right| \cdot | \hmu_{a\ell}(z) - \mu_{a\ell} | 
 \end{align*}
 where the last line holds by continuity of $\nabla^2 \psi_\ell^*$ and because $| \hmu_{a\ell}(z) - \mu_{a\ell} | = o(1)$. 
 By assumption, $\left| \nabla^2 \psi_{\ell}^*(\mu_{a\ell}) \right| = O(1)$ and hence $F(X,\mu,z) = o\left( d \Delta_{\mu, \infty} \right)$. 

Finally, on the event $\cE$, we have for all $z \in \cZ_{\tau}$, 
\begin{align*}
 G(X,\mu,\sigma,z)
 & \wle 2 d \omega_n \sqrt{n} \left( \sqrt{\frac{k}{n}} + \frac{ 2k \Delta_{\mu,\infty} }{ \alpha n } \ham(z^*, z) + C' \sqrt{ \frac{ k \ham(z^*, z) }{ n } } \right) \max_{\substack{ a\in[k] \\ \ell \in [d] } } |\nabla^2(\mu_{a\ell})| \\
 & \weq O\left( \omega_n d \sqrt{k} \left( 1 + \frac{\Delta_{\mu,\infty} }{ \sqrt{nk^{-1}} } \right) \right) \ham(z^*, z). 
\end{align*}
To finish the proof, we proceed in the same way as at the end of the proof of Theorem~\ref{thm:LaplaceMixtureRecovery}.

\subsection{Additional Lemmas}

\begin{lemma}[Generalized triangle inequality for Bregman divergences]
\label{lemma:difference_breg_divergences}
 Let $x,y,z \in \R^d$. Then
 \begin{align*}
    \dbreg_{\phi} (x,y) - \dbreg_{\phi}(x,z) \weq \dbreg_\phi(z,y) + \langle x-z, \nabla\phi(z) - \nabla\phi(y) \rangle.
 \end{align*}
\end{lemma}

\begin{proof}
By definition of Bregman divergence, 
\begin{align*}
 \dbreg_\phi(x,y) \weq \phi(x) - \phi(y) - \langle x-y, \nabla \phi(y) \rangle, \\
 \dbreg_\phi(x,z) \weq \phi(x) - \phi(z) - \langle x-z, \nabla \phi(z) \rangle,
\end{align*}
and the result holds by direct computation.
\end{proof}

\begin{lemma}
\label{lemma:bound_bregman_sequenceConverging}
 Let $\phi \colon \Theta \to \R$ be a convex twice continuously differentiable function defined on an open space $\Theta$, and let $(x_t) \in \Theta^{\Z_+}$ be a sequence such that $\lim_{t \to \infty} x_t = x$. Then, we have for $n$ large enough 
 \begin{align*}
    \left| \dbreg_{\phi}(x, x_t) \right| \wle 2 \left\| \nabla^2 \phi(x) \right\| \cdot \left\| x-x_t \right\|^2. 
 \end{align*}
\end{lemma}
\begin{proof}
 Because $\dbreg_{\phi}(x,x_t)$ equals the difference between $\phi$ evaluated at $x$ and its first order Taylor approximation around $x_t$ evaluated at $x$, we have 
 \begin{align*}
    \left| \dbreg_{\phi}(x, x_t) \right| \wle \left\| \nabla^2 \phi(x_t) \right\| \cdot \|x-x_t\|^2.
 \end{align*}
 We finish the proof by noticing that, for $t$ large enough, we have $\left\| \nabla^2 \phi(x_t) \right\| \le 2 \left\| \nabla^2 \phi(x) \right\|$ by continuity of $\nabla^2 \phi$. 
\end{proof}